\documentclass[11pt,a4paper]{amsart}
\usepackage[utf8]{inputenc}
\usepackage[english]{babel}
\usepackage{amsmath}
\usepackage{amsfonts}
\usepackage{amssymb}
\usepackage{xcolor}
\usepackage{graphicx}
\usepackage{float}
\usepackage{mathtools}
\usepackage{fullpage}
\usepackage{hyperref}

\newcommand{\R}{\mathbb R}
\newcommand{\E}{\mathbb E}
\newcommand{\Z}{\mathbb Z}
\newcommand{\N}{\mathbb N}
\renewcommand{\L}{\mathcal L}
\newcommand{\eps}{\varepsilon}
\newcommand{\1}{\mathbf 1}
\newcommand{\U}{\mathcal U}

\newtheorem{thm}{Theorem}[section]
\newtheorem{lemma}[thm]{Lemma}
\newtheorem{defn}[thm]{Definition}
\newtheorem{prop}[thm]{Proposition}
\newtheorem{cor}[thm]{Corollary}

\theoremstyle{remark}
\newtheorem{rem}[thm]{Remark}

\newenvironment{proofERRW}{\paragraph{\textit{Proof of Proposition \ref{PropConvERRW}.}}}
{\hfill$\square$}

\numberwithin{equation}{section}

\author{Titus Lupu \and Christophe Sabot \and Pierre Tarrès}
\address {CNRS and LPSM, UMR 8001,
Sorbonne Université,
4 place Jussieu,
75252 Paris cedex 05,
France}
\email
{titus.lupu@upmc.fr}

\address {
Institut Camille Jordan,
Université Lyon 1, 
43 bd. du 11 nov. 1918,
69622 Villeurbanne cedex,
France}
\email
{sabot@math.univ-lyon1.fr}

\address {NYU-ECNU Institute of Mathematical Sciences at NYU Shanghai, China; Courant Institute of Mathematical Sciences, New York, USA; CNRS and Universit\'e Paris-Dauphine, PSL Research University, Ceremade, Paris, France}
\email
{tarres@nyu.edu}

\title{Fine mesh 
limit of the VRJP in dimension one and Bass-Burdzy flow}

\begin{document}

\begin{abstract}

We introduce a continuous space limit of the Vertex Reinforced Jump Process (VRJP)  in dimension one, which we call Linearly Reinforced Motion (LRM) on $\R$. It is constructed out of a convergent Bass-Burdzy flow. The proof goes through the representation of the VRJP as a mixture of Markov jump processes. As a by-product this gives a representation in terms of a mixture of diffusions of the LRM and of the Bass-Burdzy flow itself. We also show that our continuous space limit can be obtained out of the
Edge Reinforced Random Walk (ERRW), since the ERRW and the VRJP are known to be closely related.
Compared to the discrete space processes, the LRM has an additional symmetry in the initial local times (initial occupation profile): changing them amounts to a deterministic change of the space and time scales.

\end{abstract}


\maketitle

\section{Introduction and presentation of results}
\label{SecIntro}

Let $G=(V,E,\sim)$ be an electrical network with positive conductances $(C_e)_{e\in E}$, and let 
$(\phi_{i})_{i\in V}$
be positive weights on the vertices $V$. The Vertex-Reinforced Jump Process (VRJP) is a continuous-time process $(\zeta_t)_{t\ge 0}$ taking values in $V$ which, conditionally on the past at time $t$, jumps from a vertex $i\in V$ to $j\sim i$ at rate
\begin{equation*}
C_{ij}
\Big(
\phi_{j}+
\int_0^{t}\1_{\{\zeta_s=j\}}\,ds
\Big),
\end{equation*}
where 
\begin{equation*}
\phi_{j}+
\int_0^{t}\1_{\{\zeta_s=j\}}\,ds
\end{equation*}
is the local time at vertex $j$ at time $t$, with the convention that the initial local time at $j$ is $\phi_{j}$.

The VRJP was introduced by Davis and Volkov \cite{DavisVolkov2002VRJP,DavisVolkov2004VRJP} and is closely related to the Edge-Reinforced Random Walk (ERRW) introduced by Coppersmith and Diaconis in 1986 \cite{coppersmith}, and to the supersymmetric hyperbolic model in quantum field theory, see \cite{SabotTarres2015VRJPSSHSM,dsz}; see \cite{DavisDean2010VRJP1D,bs} for more references on  the VRJP.

Our aim is to introduce a fine mesh limit of the VRJP on the one-dimensional lattice $2^{-n}\Z$ when $n$ tends to infinity. 
We start with a function $L_{0}:\R\rightarrow (0,+\infty)$, which will correspond to initial local times of the fine mesh limit,  such that 
\begin{equation}
\label{EqCond}
\int_{0}^{+\infty}L_{0}(x)^{-2} dx
=\int_{-\infty}^{0}L_{0}(x)^{-2} dx
=+\infty.
\end{equation}
As we will see further, \eqref{EqCond} is a condition for non-explosion to infinity.

We define $(X^{(n)}_{t})_{t\geq 0}$ as the continuous-time VRJP started from $0$ on the network $2^{-n}\Z$, with 
uniform conductances
$C_e=C=2^{2n-1}$, and 
$\phi_{i2^{-n}}=2^{-n}L_0(i2^{-n})$.
We define its local time as
\begin{displaymath}
\ell^{(n)}_{t}(x)=
2^{n}\int_{0}^{t}\1_{X^{(n)}_{s}=x}ds,
~x\in 2^{-n}\Z.
\end{displaymath}
The factor $2^{n}$ is the inverse of the size of a cell around a vertex.
The jump rates at time $t$ from 
$x$ to $x+\sigma 2^{-n}$, $\sigma\in\lbrace -1,1\rbrace$, are
\begin{equation}
\label{EqJumpRate}
2^{2n-1}L^{(n)}_{t}(x+\sigma 2^{-n}),
\end{equation}
with
\begin{displaymath}
L^{(n)}_{t}=L^{(n)}_{0}+\ell^{(n)}_{t},
\end{displaymath}
where $L^{(n)}_{0}$ is the restriction to $2^{-n}\Z$ of the initial occupation profile $L_0$. The process is defined up to a time
$t^{(n)}_{\rm max}\in (0,+\infty]$, as it might reach
$-\infty$ or $+\infty$ in finite time.

We are interested in the limit in law of
$((X^{(n)}_{t})_{0\leq t\leq t^{(n)}_{\rm max}},
(L^{(n)}_{t}(x))_{x\in 2^{-n}\Z, 0\leq t\leq t^{(n)}_{\rm max}})$ as 
$n\to +\infty$.

The order in the conductances $C$ and initial local times
$\phi$ yields, up to a linear change of time, the only interesting limit, i.e. which is not Brownian motion or a constant process.

We will denote the limit process on $\R$ by
$((X_{t})_{t\geq 0},
(L_{t}(x))_{x\in \R, t\geq 0})$. One can construct it out of the flow of solutions to the Bass-Burdzy equation:
\begin{equation}
\label{EqBassBurdzy}
\dfrac{d Y_{u}}{du}=
\left\lbrace
\begin{array}{ll}
-1 & \text{if}~Y_{u}>B_{u},\\ 
1 & \text{if}~Y_{u}<B_{u},
\end{array} 
\right.
\end{equation}
where $(B_{u})_{u\geq 0}$ is the standard Brownian motion on $\R$ started from $0$. Bass and Burdzy showed in \cite{BassBurdzy99StochBiff} that
\eqref{EqBassBurdzy} has a.s., for a given initial condition, a unique solution which is Lipschitz continuous. Let us explain how this equation naturally appears in our context.

Assume first that there is no reinforcement, that is to say
$L^{(n)}_{t}$ is replaced by $L_{0}$ in the jump rates of
\eqref{EqJumpRate}. Then the processes would converge to a Markov diffusion with the infinitesimal generator
\begin{displaymath}
\dfrac{1}{2}L_{0}(x)\dfrac{d^{2}}{dx^{2}}+
L_{0}(x)\Big(\dfrac{d}{dx}\big(\log(L_{0}(x))\big)\Big)\dfrac{d}{dx}.
\end{displaymath}
So if one does a change of scale
\begin{displaymath}
dy=L_{0}(x)^{-2} dx
\end{displaymath}
(by the way, this is where the condition
\eqref{EqCond} comes from), and a change of time
\begin{displaymath}
du= L_{0}(X_{t})^{-3}dt,
\end{displaymath}
where $X_{t}$ is the position of the particle at time $t$, one gets a Brownian motion. 
See Section 4.1 in \cite{ItoMcKean1974Diffusions},
Sections 16.5 and 16.6 in\cite{Breiman1992Probability}, and
Sections VII.2 and VII.3 in \cite{RevuzYor1999BMGrundlehren}
for the notions of natural scale and natural speed measures of one-dimensional diffusions.

Now assume that we do have a reinforcement and that there is some limit process $(X_{t})_{t\geq 0}$, with occupation densities
$L_{t}-L_{0}$. Then one would like to have a dynamical change of scale
\begin{displaymath}
d S_{t}(x)= L_{t}(x)^{-2} dx,
\end{displaymath}
such that $(S_{t}(X_{t}))_{t\geq 0}$ is a martingale
(which corresponds to choosing
$S^{-1}_{t}(0)$ in an appropriate way), and such that after a change of time
\begin{equation}
\label{EqTS}
du=L_{t}(X_{t})^{-3} dt,
\end{equation}
this martingale becomes a Brownian motion
$B_{u}=S_{u}(X_{u})$. This corresponds to the idea that 
after time $t$, $X_{t+\Delta t}$, behaves, for
$\Delta t \ll 1$, almost like a diffusion with the infinitesimal generator
\begin{displaymath}
\dfrac{1}{2}L_{t}(x)\dfrac{d^{2}}{dx^{2}}+
L_{t}(x)\dfrac{d}{dx}(\log(L_{t}(x)))\dfrac{d}{dx}.
\end{displaymath}
Given $x_{1}<x_{2}\in\R$ fixed, 
in the time scale \eqref{EqTS}, we have that
\begin{eqnarray*}
\dfrac{d}{d u}(S_{u}(x_{2})-S_{u}(x_{1}))&=&
\dfrac{dt}{du}\dfrac{d}{d t}(S_{t}(x_{2})-S_{t}(x_{1}))=
L_{t}(X_{t})^{3}\dfrac{d}{dt}
\int_{x_{1}}^{x_{2}}L_{t}(x)^{-2}dx
\\&=&-2L_{t}(X_{t})^{3}
\int_{x_{1}}^{x_{2}}L_{t}(x)^{-3}d_{t}L_{t}(x)dx
= -2\int_{x_{1}}^{x_{2}}d_{t}L_{t}(x)dx
\\&=&-2\1_{x_{1}<X_{t}<x_{2}}=
-2\1_{S_{u}(x_{1})<B_{u}<S_{u}(x_{2})}.
\end{eqnarray*}
If we moreover take into account that after time $t$, 
$X_{t+ \Delta t}$ should spend infinitesimally the same amount of 
time left and right from $X_{t}$, we get the equation
\begin{displaymath}
\dfrac{d}{d u}(S_{u}(x))=
-\1_{S_{u}(x)>B_{u}}+
\1_{S_{u}(x)<B_{u}},
\end{displaymath}
which is exactly that of \eqref{EqBassBurdzy}.

We will "reverse-engineer" the above construction.
Let $(\Psi^{B}_{u}(y))_{y\in\R, u\geq 0}$ be the flow of solutions to
\eqref{EqBassBurdzy}. $u\mapsto\Psi^{B}_{u}(y)$ is the Lipschitz solution to
\eqref{EqBassBurdzy} with initial condition $Y_{0}=y$. We call $\Psi^{B}$ the \textbf{convergent Bass-Burdzy} flow. It is a flow of diffeomorphisms of $\R$
\cite{BassBurdzy99StochBiff}. Let $\xi_{u}$ be
\begin{displaymath}
\xi_{u}= (\Psi^{B}_{u})^{-1}(B_{u}).
\end{displaymath}
The process $(\xi_{u})_{u\geq 0}$ has a time-space continuous family of local times $(\Lambda_{u}(y))_{y\in\R, u\geq 0}$
\cite{HuWarren00BBFlow}, such that for all 
$f:\R\rightarrow \R$ bounded, Borel measurable, and all $u\geq 0$,
\begin{displaymath}
\int_{0}^{u}f(\xi_{v}) dv=
\int_{\R}f(y)\Lambda_{u}(y) dy.
\end{displaymath}
Moreover, $\Lambda_{u}(y)\leq 1/2$.

\begin{defn}
\label{DefLSRM}
Let $L_{0}:\R\rightarrow (0,+\infty)$ be a continuous function. Moreover, we assume that the condition \eqref{EqCond} is satisfied. 
Let $x_{0}\in\mathbb{R}$. Denote, for $x\in\R$,
\begin{equation}
\label{EqS0}
S_{0}(x)=\int_{x_{0}}^{x}L_{0}(r)^{-2} dr.
\end{equation}
Perform the change of time
\begin{equation}
\label{EqTimeChange}
dt=L_{0}(S_{0}^{-1}(\xi_{u}))^{3}(1-2\Lambda_{u}(\xi_{u}))^{-\frac{3}{2}} du.
\end{equation}
The process
$(S_{0}^{-1}(\xi_{u(t)}))_{t\geq 0}$, where $u(t)$ is the inverse time change of \eqref{EqTimeChange}, is called the 
\textbf{Linearly Reinforced Motion (LRM)} starting from $x_{0}$, with \textbf{initial occupation profile} $L_{0}$.
We call $(\xi_{u})_{u\geq 0}$ the corresponding \textbf{reduced process}
and $(B_{u})_{u\geq 0}$ the corresponding 
\textbf{driving Brownian motion}. Set
\begin{equation}
\label{EqLt}
L_{t}(x)=L_{0}(x)(1-2\Lambda_{u(t)}(S_{0}(x)))^{-\frac{1}{2}}.
\end{equation}
$(L_{t}(x))_{x\in\R}$ is the \textbf{occupation profile at time} $t$.
\end{defn}

\begin{rem}
\label{RemTimeChange}
The time change \eqref{EqTimeChange} is \textit{a posteriori}
\begin{displaymath}
dt = L_{t}(X_{t})^{3} du.
\end{displaymath}
\end{rem}

\begin{thm}
\label{ThmMain}
The VRJP process jointly with its occupation profiles 
\\A$(X^{(n)}_{t},L^{(n)}_{t}(x))
_{x\in 2^{-n}\Z, 0\leq t\leq t^{(n)}_{\rm max}}$
converge in law as $n\to +\infty$ to a Linearly Reinforced Motion started from $0$ and its occupation profiles 
$(X_{t},L_{t}(x))_{x\in \R, t\geq 0}$.
The topology of the convergence is that of uniform convergence on compact subsets.
In particular $t^{(n)}_{\rm max}$ converges 
in probability to $+\infty$.
The spatial processes are considered to be interpolated linearly outside $2^{-n}\Z$.
\end{thm}

\begin{rem} Previously, a different Bass-Burdzy flow appeared in the study of continuous self-interacting processes. In \cite{Warren05BBFlow} it was shown that the flow of solutions to 
\begin{displaymath}
\dfrac{dY_{u}}{du} = \1_{Y_{u}>B_{u}}
\end{displaymath}
was related to the Brownian first passage bridge conditioned by its family of local times and to the Brownian burglar \cite{WarrenYor1998Burglar}.
\end{rem}

The LRM has a symmetry property under the change of the initial occupation profile. It is a straightforward consequence of Definition
\ref{DefLSRM}. One uses the same driving Brownian motion and reduced process. This symmetry also implies a scaling property, when additionally to the space and time, one also scales the initial occupation profile. We state this next. 

\begin{prop}
\label{PropSym}
(1) Let $(\chi_{\tau})_{\tau\geq 0}$ be the Linearly Reinforced Motion starting from $0$, with initial occupation profile $1$. Given
$x_{0}\in\R$ and another occupation profile $L_{0}$, define the change of time
\begin{displaymath}
dt = L_{0}(S_{0}^{-1}(\chi_{\tau}))^{3} d\tau,
\end{displaymath}
and consider the change of scale $S_{0}$ given by
\eqref{EqS0}. Then
$X_{t}=S_{0}^{-1}(\chi_{\tau(t)})$ is 
a Linearly Reinforced Motion starting from $x_{0}$, with
initial occupation profile $L_{0}$.

(2) Consequently, if $(X_{t})_{t\geq 0}$ is an LRM starting from
$0$ with initial occupation profile $L_{0}$, and $c>0$ is a constant, then $(c^{2}X_{c^{-3}t})_{t\geq 0}$ is an LRM with initial occupation profile $c L_{0}$.
\end{prop}

It was shown in \cite{SabotTarres2015VRJPSSHSM} that on any electrical network, the VRJP has the same law as a time-change of a mixture of Markov (non-reinforced) jump processes. In our setting, the random environment related to the VRJP converges. This gives us in the limit a description of the LRM as a time-changed diffusion in random environment. 

Let $S_{0}$ be the change of scale defined by \eqref{EqS0}, with
$x_{0}=0$.
Let $(W(y))_{y\geq 0}$ and $(W(-y))_{y\geq 0}$ be two independent
standard Brownian motions, started from $0$, where $y$ is seen as a space variable. We see $(W(y))_{y\in\R}$ as a Brownian motion parametrized by
$\R$. Define
\begin{equation}
\label{EqU}
\U(x)= \sqrt{2}W\circ S_{0}(x) + \vert S_{0}(x)\vert.
\end{equation}
Consider $(Z_{q})_{q\geq 0}$ the diffusion in random potential
$2\U - 2\log(L_{0})$. Conditional on $(\U(x))_{x\in\R}$, it is a Markov diffusion on $\R$, started from 
$Z_{0}=0$, with the infinitesimal generator
\begin{equation}
\label{EqGenMix}
\dfrac{1}{2}\dfrac{d^{2}}{dx^{2}}
+
\Big(\dfrac{d}{dx}\big(\log(L_{0}(x))-\U(x)\big)\Big)
\dfrac{d}{dx}.
\end{equation}
We will denote by $(\lambda_{q}(x))_{x\in\R, y\geq 0}$ the family of local times of
$(Z_{q})_{q\geq 0}$.

Although the function $x\mapsto\log(L_{0}(x))-\U(x)$ is in general not differentiable, the diffusion $(Z_{q})_{q\geq 0}$ is well defined.
For that, consider the natural scale function 
\begin{equation}
\label{EqNatScal}
\mathcal{S}(x)=\int_{0}^{x}L_{0}(r)^{-2}e^{2\U(r)} dr.
\end{equation}
The condition \eqref{EqCond} and the fact that $\U$ is a.s. bounded from below imply that 
\begin{displaymath}
\text{a.s.}~~\mathcal{S}(-\infty)= -\infty,
\qquad
\mathcal{S}(+\infty)= +\infty.
\end{displaymath}
$(\mathcal{S}(Z_{q}))_{q\geq 0}$ is a local martingale and a Markov diffusion with
infinitesimal generator
\begin{displaymath}
\dfrac{1}{2}(\mathcal{S}'\circ\mathcal{S}^{-1}(\varsigma))^{2}
\dfrac{d^{2}}{d\varsigma^{2}}.
\end{displaymath}
It is a time-changed Brownian motion, and in particular, it is defined up to $q=+\infty$.
In the particular case $L_{0}\equiv 1$, the generator 
\eqref{EqGenMix} is equal to
\begin{displaymath}
\dfrac{1}{2}\dfrac{d^{2}}{dy^{2}}
-
\sqrt{2}\Big(\dfrac{d}{dy} W(y)\Big)
\dfrac{d}{dy}
-\operatorname{sgn}(y)\dfrac{d}{dy},
\end{displaymath}
$\big(\frac{d}{dy} W(y)\big)_{y\in\R}$ being the white noise.
For some background on diffusions in random Wiener potential, we refer to
\cite{Schumacher8595RandomDiffus, Brox86Wiener, Tanaka95RandomDiffus} and the references therein.

\begin{thm}
\label{ThmMixture}
The Linearly Reinforced Motion $(X_{t})_{t\geq 0}$,
started from $0$, with initial occupation profile $L_{0}$, has the same law as a time-change of the mixture of diffusions
$(Z_{q})_{q\geq 0}$, where the time-change is given by
\begin{equation}
\label{EqTC1}
dt=(L_{0}(Z_{q})^{2}+2\lambda_{q}(Z_{q}))^{-\frac{1}{2}} dq.
\end{equation}
\end{thm}

\begin{rem}
The mixture of diffusions $(Z_{q})_{q\geq 0}$ is itself a reinforced process. Informally, one can imagine it as having a time-dependent infinitesimal generator
\begin{displaymath}
\dfrac{1}{2}\dfrac{d^{2}}{dx^{2}}+\dfrac{1}{2}
\dfrac{d}{dx}(\log(L_{0}(x)^{2}+2\lambda_{q}(x)))\dfrac{d}{dx}.
\end{displaymath}
\end{rem}

We will prove Theorem \ref{ThmMixture} by constructing out of the VRJP a discrete analogue of the convergent Bass-Burdzy flow.

Theorem \ref{ThmMixture} has an immediate implication on the reduced process
$(\xi_{u})_{u\geq 0}$.

\begin{cor}
\label{CorXiMelange}

Let $\xi_{u}=(\Psi_{u}^{B})^{-1}(B_{u})$ be the reduced process obtained out of the Bass-Burdzy flow $(\Psi_{u}^{B})_{u\geq 0}$.
Let $(\bar{Z}_{\bar q})_{\bar q\geq 0}$ be a process, 
that conditional on
$(W(y))_{y\in\R}$ is a Markov diffusion with generator
\begin{displaymath}
\dfrac{1}{2}\dfrac{d^{2}}{dy^{2}}
-
\sqrt{2}\Big(\dfrac{d}{dy} W(y)\Big)
\dfrac{d}{dy}
-\operatorname{sgn}(y)\dfrac{d}{dy},
\end{displaymath}
and $(\bar{\lambda}_{\bar q}(y))_{y\in\R, \bar q\geq 0}$
its family of local times.
Let be the time change
\begin{displaymath}
du = (1+2\bar{\lambda}_{\bar q}(\bar{Z}_{\bar q}))^{-2} d \bar q.
\end{displaymath}
Then the time changed process
$(\bar{Z}_{\bar{q}(u)})_{u\geq 0}$ has the same law as
$(\xi_{u})_{u\geq 0}$. Moreover, in this construction of
$(\xi_{u})_{u\geq 0}$, we have the following relation between
the local times:
\begin{displaymath}
\Lambda_{u}(y)=\dfrac{\bar{\lambda}_{\bar q}(y)}
{1+2\bar{\lambda}_{\bar q}(y)}.
\end{displaymath}
\end{cor}

Next table sums up the correspondences between different processes, an LRM with initial occupation profile $L_{0}$, the LRM with initial occupation profile $1$, denoted $(\chi_{\tau})_{\tau\geq 0}$, the reduced process $(\xi_{u})_{u\geq 0}$, and the diffusion in random environment $(\bar Z_{\bar q})_{\bar q\geq 0}$. On the rows with "correspondence", all the quantities are equal.

\medskip

\begin{center}

\begin{tabular}{|c|c|c|c|c|}
\hline
Process & $X_{t}$ & $\chi_{\tau}$ & $\xi_{u}$ & 
$\bar Z_{\bar q}$ \\
\shortstack{Description\\\vphantom{a}} & \shortstack{LRM, initial\\occup. profile $L_{0}$} & \shortstack{LRM, initial\\occup. profile $1$} & 
\shortstack{
Reduced \\ 
\vphantom{$L_{0}$}process}
& \shortstack{Diffusion in\\random environment} \\  
\hline 
Space variable & $x$ & $y$ & $y$ & $y$ \\ 

Time variable & $t$ & $\tau$ & $u$ & $\bar q$ \\ 

	Local time & $L_{t}(x)-L_{0}(x)$ & 
	$L^{\chi}_{\tau}(y)-1$ & $\Lambda_{u}(y)$ & 
	 $\bar\lambda_{\bar q}(y)$ \\
	
\shortstack{Space\\correspondence} & $L_{0}(x)^{-2} dx$ & $dy$ & $dy$ & $dy$ \\ 

\shortstack{Time\\correspondence} & $L_{t}(X_{t})^{-3} dt$ & 
$L_{\tau}(\chi_{\tau})^{-3} d\tau$ & $du$ & 
$(1+2\bar{\lambda}_{\bar q}(\bar{Z}_{\bar q}))^{-2} d \bar q$ \\ 

\shortstack{Local time\\correspondence} & 
$\dfrac{1}{2}\bigg(1-\dfrac{L_{0}(x)^{2}}{L_{t}(x)^{2}}\bigg)$ & 
$\dfrac{1}{2}(1-L^{\chi}_{\tau}(y)^{-2})$ & $\Lambda_{u}(y)$ & 
$\dfrac{\bar{\lambda}_{\bar q}(y)}
{1+2\bar{\lambda}_{\bar q}(y)}$ \\ 
\hline 
\end{tabular}
\end{center}

\bigskip

The convergence of the VRJP to a continuous space process has a version for the Edge Reinforced Random Walk.  For references on the ERRW see
\cite{diaconis,keane-rolles,diaconis-rolles,rolles2,merkl-rolles1,ack}. 
It was shown in \cite{SabotTarres2015VRJPSSHSM} that an ERRW has same distribution as the discrete-time process of a VRJP in a network with random conductances, hence it is a mixture of Markovian random walks.

In our context, we consider a discrete time reinforced walk 
$(\widehat{Z}^{(n)}_{k})_{k\geq 0}$ on $2^{-n}\Z$, started at $0$. The weight of an edge
$w_{k}^{(n)}(x,x+2^{-n})$ at time $k$ will be
\begin{multline*}
w_{k}^{(n)}(x-2^{-n},x)= w_{k}^{(n)}(x,x-2^{-n})\\= w_{0}^{(n)}(x-2^{-n},x)+
\operatorname{Card}\lbrace j\in\lbrace 1,\dots,k\rbrace\vert
\lbrace \widehat{Z}^{(n)}_{j-1}, \widehat{Z}^{(n)}_{j}\rbrace =\lbrace x-2^{-n},x\rbrace\rbrace,
\end{multline*}
where $\lbrace\cdot,\cdot\rbrace$ stands for the undirected edge and
$w_{0}^{(n)}(x-2^{-n},x)\in (0,+\infty)$. The transition probabilities are:
\begin{displaymath}
\mathbb{P}(\widehat{Z}^{(n)}_{k+1}= x +\sigma 2^{-n}\vert
\widehat{Z}^{(n)}_{k}=x, (\widehat{Z}^{(n)}_{j})_{0\leq j\leq k})=
\dfrac{w_{k}^{(n)}(x,x +\sigma 2^{-n})}
{w_{k}^{(n)}(x,x - 2^{-n})+w_{k}^{(n)}(x,x + 2^{-n})},
~~\sigma\in\{-1,1\}.
\end{displaymath}
For initial weights we will take
\begin{displaymath}
w_{0}^{(n)}(x-2^{-n},x) = 2^{n-1} L_{0}(x-2^{-n})L_{0}(x).
\end{displaymath}

\begin{prop}
\label{PropConvERRW}
The process
$(\widehat{Z}^{(n)}_{\lfloor 4^{n} q\rfloor}
)_{q\geq 0}$
convergences in law as $n\to +\infty$ towards
$(Z_{q})_{q\geq 0}$,
the mixture of diffusions of Theorem
\ref{ThmMixture}, with the infinitesimal generator given by \eqref{EqGenMix}.
\end{prop}

\begin{rem}
The fact that the ERRW has a fine mesh limit which is a diffusion in random potential is reminiscent of the Sinai's random walk \cite{Sinai82RW} converging to a Brox diffusion \cite{Brox86Wiener,Seignourel2000SinaiBrox,Pacheco2016SinaiBrox}. 
In the Brox diffusion however the random potential contains only a Wiener term and no drift as in our case. See also 
\cite{Davis1996CPY} for the once-reinforced random walk converging to the Carmona-Petit-Yor process \cite{CPY1998}.
\end{rem}

\medskip

Our paper is organized as follows. In Section \ref{SecBassBurdzy} we will recall some properties of Bass-Burdzy flows and see what it implies for the Linearly Reinforced Motion. 
In Section \ref{SecEnvir} we will show the convergence of the random environment related to the VRJP, and as a consequence the convergence of the VRJP to a mixture of time-changed diffusions. We will also recall the random environment associated to the ERRW and deduce the convergence of the ERRW. In Section \ref{SecConv}, we will show that this mixture of time-changed diffusions coincides with the Linearly Reinforced Motion,
thus concluding the proofs of Theorems \ref{ThmMain} and
\ref{ThmMixture}. We will deduce a couple of consequences of this, such as the long-time behaviour of the LRM. In our paper we will use different time scales, $t$, $\tau$, $u$, $q$, $\bar{q}$, etc., and the notations like $q(t)$ will denote the changes of time that transform one time scale into an other.

Next is a simulation of the Linearly Reinforced Motion with initial occupation profile constant equal to $1$, on the time interval $[0,8]$. The simulation is obtained by running a VRJP on a fine lattice.
On this picture one observes the emergence of a continuous stochastic process in the fine mesh limit. One also 
sees the difference with a Brownian motion. Indeed, one distinguishes significant reinforcement between the levels $0$ and $2$.

\begin{figure}[H]	   
	\centering
	\includegraphics[width=3.4in]{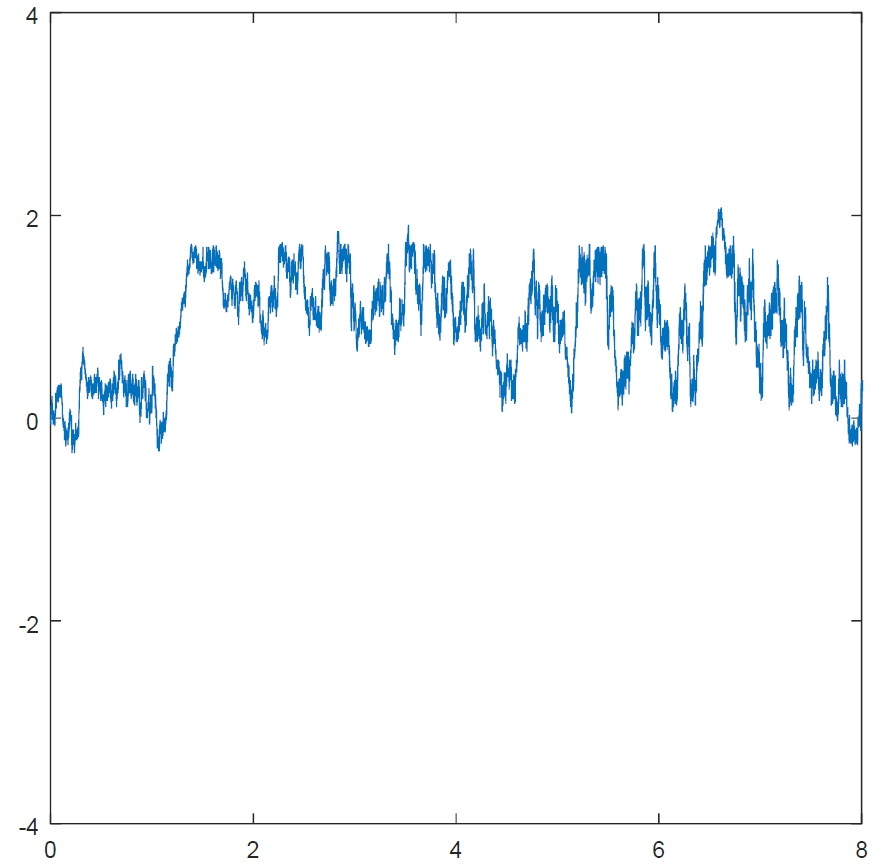}
	\caption{LRM with $L_{0}\equiv 1$ on time-interval $[0,8]$.}
	\label{Fig1}	
\end{figure}

\section{Convergent Bass-Burdzy flow and Linearly Reinforced Motion}
\label{SecBassBurdzy}

Let $(B_{u})_{u\geq 0}$ be a standard Brownian motion on
$\R$, starting at $0$, and let $(\mathcal{F}^{B}_{u})_{u\geq 0}$ be the associated filtration. For $u_{0}\geq 0$, we will denote
\begin{displaymath}
(B\circ\theta_{u_{0}})_{u}=B_{u_{0}+u}-B_{u_0}.
\end{displaymath}

We consider the differential equation
\eqref{EqBassBurdzy}:
\begin{equation*}
\dfrac{d Y_{u}}{du}=
\left\lbrace
\begin{array}{ll}
-1 & \text{if}~Y_{u}>B_{u},\\ 
1 & \text{if}~Y_{u}<B_{u},
\end{array} 
\right.
\end{equation*}
with some initial condition $Y_{0}=y\in\R$.
$(B_{u}-\Psi^{B}_{u}(y))_{y\in\R,u\geq 0}$ is a stochastic flow, solution to the SDE
\begin{equation}
\label{EqBassBurdzySDE}
d\zeta_{u}=dB_{u} -(-\1_{\zeta_{u}>0} + \1_{\zeta_{u}<0})du.
\end{equation}
The equation\eqref{EqBassBurdzySDE} falls into the class studied in
\cite{Attanasio10DiscDrift} (bounded variation drift).
Next we list the main results on the solutions to \eqref{EqBassBurdzy}.

\begin{prop}[Bass-Burdzy \cite{BassBurdzy99StochBiff}, 
Hu-Warren \cite{HuWarren00BBFlow}, Attanasio \cite{Attanasio10DiscDrift}]
\label{PropBBFlow}
For every initial condition 
$Y_{0}=y\in \R$, there is a.s. a unique solution to
\eqref{EqBassBurdzy} which is Lipschitz continuous. We denote it
$(\Psi^{B}_{u}(y))_{u\geq 0}$. For any 
$y_{1},\dots,y_{k}\in\R$, the joint law of
$(B_{u},\Psi^{B}_{u}(y_{1}),\dots,\Psi^{B}_{u}(y_{k}))_{u\geq 0}$
is uniquely determined.
One can construct
$(\Psi^{B}_{u}(y))_{y\in\R,u\geq 0}$ simultaneously for all
$y\in\R$ such that $(y,u)\mapsto \Psi^{B}_{u}(y)$ is continuous on
$\R\times [0,+\infty)$. Moreover, we have the following properties:
\begin{enumerate}
\item The flow $(\Psi^{B}_{u}(y))_{y\in\R,u\geq 0}$ is adapted to the filtration $(\mathcal{F}^{B}_{u})_{u\geq 0}$.
\item (Strong Markov property). For any $U_{0}$ stopping time for 
$(\mathcal{F}^{B}_{u})_{u\geq 0}$,
\begin{displaymath}
\Psi^{B}_{U_{0}+u}(y)=
\Psi^{B\circ\theta_{U_{0}}}_{u}(y-B_{U_{0}})+B_{U_{0}}.
\end{displaymath}
\item A.s., for any $\alpha\in(0,1/2)$ and for all $u\geq 0$, $y\mapsto \Psi^{B}_{u}(y)$ is
a $\mathcal{C}^{1,\alpha}$-diffeomorphism of $\R$.
That is to say, $y\mapsto \Psi^{B}_{u}(y)$ is an increasing bijection,
$\dfrac{\partial}{\partial y} \Psi^{B}_{u}(y)$ is positive on $\R$, and
both the functions
$y\mapsto \dfrac{\partial}{\partial y} \Psi^{B}_{u}(y)$ and
$y\mapsto \dfrac{\partial}{\partial y} (\Psi^{B}_{u})^{-1}(y)$ are
locally $\alpha$-Hölder continuous. 
\item The process $(B_{u}-\Psi^{B}_{u}(y))_{u\geq 0}$ admits semi-martingale local times at level 0, $(\L_{u}(y))_{u\geq 0}$,
\begin{displaymath}
\L_{u}(y)=\lim_{\eps\to 0}\dfrac{1}{2\eps}\int_{0}^{u}
\1_{\vert B_{u}-\Psi^{B}_{v}(y)\vert<\eps} dv,
\end{displaymath}
such that the map $(y,u)\mapsto\L_{u}(y)$ is continuous.
\item For the space derivative of the flow, one has
\begin{equation}
\label{EqSpaceDeriv}
\dfrac{\partial}{\partial y} \Psi^{B}_{u}(y)=
\exp(-2\L_{u}(y)).
\end{equation}
\item The process $((\Psi^{B}_{u})^{-1}(B_{u}))_{u\geq 0}
=(\xi_{u})_{u\geq 0}$ admits occupation densities (local times) 
\\$(\Lambda_{u}(y))_{y\in\R,u\geq 0}$, continuous in $(y,u)$.
Moreover, the following identity holds:
\begin{equation}
\label{EqIDLocTime}
\Lambda_{u}(y)=\dfrac{1}{2}(1-\exp(-2\L_{u}(y))).
\end{equation}
In particular, $\Lambda_{u}(y)\leq 1/2$.
\item The process $(\xi_{u})_{u\geq 0}$ is recurrent,
that is to say, for all $u_{0}\geq 0$, the process will visit a.s. all points after $u_{0}$.
\end{enumerate}
\end{prop}

Next we show some elementary properties of $(\xi_{u})_{u\geq 0}$ which we did not find as such in our references
\cite{BassBurdzy99StochBiff,HuWarren00BBFlow,Attanasio10DiscDrift}.

\begin{prop}
\label{PropAdditional}
$(\xi_{u})_{u\geq 0}$ satisfies:
\begin{enumerate}
\item A.s., for any $\alpha\in (0,1/2)$, the process
$(\xi_{u})_{u\geq 0}$ is locally $\alpha$-Hölder continuous.
\item Let $u>0$ and consider $(u_{i,j})_{0\leq j\leq N_{i},i\geq 0}$
a deterministic family such that 
\begin{displaymath}
0=u_{i,0}<u_{i,1}<\dots<u_{i,N_{i}-1}<u_{N_{i}}=u
\end{displaymath}
and
\begin{displaymath}
\lim_{i\to +\infty}\max_{1\leq j\leq N_{i}}(u_{i,j}-u_{i,j-1})=0.
\end{displaymath}
Then,
\begin{equation}
\label{EqQV}
\lim_{i\to +\infty}\sum_{j=1}^{N_{i}}(\xi_{u_{i,j}}-\xi_{u_{i,j-1}})^{2}
=\int_{0}^{u}(1-2\Lambda_{v}(\xi_{v}))^{-2} dv
\end{equation}
in probability.
\item Let $(\rho_{u})_{u\geq 0}$ be the process
\begin{displaymath}
\rho_{u}=\xi_{u}-\int_{0}^{u}
(1-2\Lambda_{v}(\xi_{v}))^{-1} dB_{v}.
\end{displaymath}
For a family $(u_{i,j})$ as above,
\begin{equation*}
\lim_{i\to +\infty}\sum_{j=1}^{N_{i}}(\rho_{u_{i,j}}-\rho_{u_{i,j-1}})^{2}
=0
\end{equation*}
in probability.
\end{enumerate}
\end{prop}

\begin{proof}
First note that for any $y\in\R$,
\begin{displaymath}
\vert (\Psi_{u}^{B})^{-1}(y)-y\vert=
\vert \Psi_{u}^{B}\circ(\Psi_{u}^{B})^{-1}(y)-
(\Psi_{u}^{B})^{-1}(y)\vert \leq u,
\end{displaymath}
and
\begin{displaymath}
\vert (\Psi_{u}^{B})^{-1}(y_{2})-(\Psi_{u}^{B})^{-1}(y_{1})\vert
\leq\exp\big(2\sup_{y\in\R}\L_{u}(y)\big)\vert y_{2}-y_{1}\vert
=\big(1-2\sup_{y\in\R}\Lambda_{u}(y)\big)^{-1}\vert y_{2}-y_{1}\vert,
\end{displaymath}
where for the second inequality we used that
\begin{displaymath}
\dfrac{\partial}{\partial y}(\Psi^{B}_{u})^{-1}(y)=
\exp(2\mathcal{L}_{u}((\Psi^{B}_{u})^{-1}(y)))=
(1-2\Lambda_{u}((\Psi^{B}_{u})^{-1}(y)))^{-1}.
\end{displaymath}
Then write
\begin{displaymath}
\xi_{u_{2}}=(\Psi_{u_{2}}^{B})^{-1}(B_{u_{2}})=
(\Psi_{u_{1}}^{B})^{-1}(
(\Psi_{u_{2}-u_{1}}^{B\circ\theta_{u_{1}}})^{-1}
(B_{u_{2}}-B_{u_{1}})+B_{u_{1}}).
\end{displaymath}
It follows that
\begin{eqnarray*}
\label{EqHolder}
\vert\xi_{u_{2}}-\xi_{u_{1}}\vert&\leq&
\big(1-2\sup_{y\in\R}\Lambda_{u_{1}}(y)\big)^{-1}
\vert(\Psi_{u_{2}-u_{1}}^{B\circ\theta_{u_{1}}})^{-1}
(B_{u_{2}} - B_{u_{1}})\vert
\\&\leq& 
\big(1-2\sup_{y\in\R}\Lambda_{u_{1}}(y)\big)^{-1}
\vert B_{u_{2}}-B_{u_{1}}\vert
\\&&+
\big(1-2\sup_{y\in\R}\Lambda_{u_{1}}(y)\big)^{-1}
\vert(\Psi_{u_{2}-u_{1}}^{B\circ\theta_{u_{1}}})^{-1}
(B_{u_{2}} - B_{u_{1}})
-(B_{u_{2}} - B_{u_{1}})\vert
\\&\leq&
\big(1-2\sup_{y\in\R}\Lambda_{u_{1}}(y)\big)^{-1}
\vert B_{u_{2}}-B_{u_{1}}\vert
+\big(1-2\sup_{y\in\R}\Lambda_{u_{1}}(y)\big)^{-1}
(u_{2}-u_{1}),
\end{eqnarray*}
which implies (1).

Let us show (2).
Refining the above computation, one gets that
\begin{align}
\label{EqIto}
\xi_{u_{i,j}}-\xi_{u_{i,j-1}}=&
(1-2\Lambda_{u_{i,j-1}}(\xi_{u_{i,j-1}}))^{-1}
(B_{u_{i,j}}-B_{u_{i,j-1}})
\\ \nonumber
&+o(\vert B_{u_{i,j}}-B_{u_{i,j-1}}\vert +(u_{i,j}-u_{i,j-1}))
+O(u_{i,j}-u_{i,j-1}),
\end{align}
where 
\begin{multline*}
\vert o(\vert B_{u_{i,j}}-B_{u_{i,j-1}}\vert
+(u_{i,j}-u_{i,j-1}))\vert\leq
(\vert B_{u_{i,j}}-B_{u_{i,j-1}}\vert + (u_{i,j}-u_{i,j-1}))
\\\times\sup_{\substack{\vert y_{1}-y_{2}\vert \leq \vert B_{u_{i,j}}-B_{u_{i,j-1}}\vert \\+2(u_{i,j}-u_{i,j-1})}}
\vert
(1-2\Lambda_{u_{i,j-1}}(y_{2}))^{-1}-
(1-2\Lambda_{u_{i,j-1}}(y_{1}))^{-1}
\vert,
\end{multline*}
and
\begin{displaymath}
\vert O(u_{i,j}-u_{i,j-1})\vert
\leq
\big(1-2\sup_{y\in\R}\Lambda_{u_{1}}(y)\big)^{-1}
(u_{2}-u_{1}).
\end{displaymath}
Thus, the sum in
\eqref{EqQV} behaves, as $i\to +\infty$, like
\begin{displaymath}
\sum_{j=1}^{N_{i}}
(1-2\Lambda_{u_{i,j-1}}(\xi_{u_{i,j-1}}))^{-2}
(B_{u_{i,j}}-B_{u_{i,j-1}})^{2}.
\end{displaymath}
To conclude, we use that
\begin{displaymath}
\lim_{i\to +\infty}
\sum_{j=1}^{N_{i}}
\vert(B_{u_{i,j}}-B_{u_{i,j-1}})^{2}-
(u_{i,j}-u_{i,j-1})\vert = 0
\end{displaymath}
in probability.

Let us show (3). 
Let $A>1$. Here $U_{A}$ will denote the stopping time
\begin{displaymath}
U_{A}=\inf\Big\{v\geq 0\Big\vert\sup_{y\in\R}
\dfrac{\partial}{\partial y}(\Psi^{B}_{v})^{-1}(y)\geq A
\Big\}.
\end{displaymath}
Then
\begin{displaymath}
\lim_{A\to +\infty}
\mathbb{P}(U_{A}>u)=1.
\end{displaymath}
We use \eqref{EqIto} and write
\begin{align*}
\rho_{u_{i,j}\wedge U_{A}}-\rho_{u_{i,j-1}\wedge U_{A}}=&
(1-2\Lambda_{u_{i,j-1}\wedge U_{A}}(\xi_{u_{i,j-1}\wedge U_{A}}))^{-1}
(B_{u_{i,j}\wedge U_{A}}-B_{u_{i,j-1}\wedge U_{A}})
\\&-\int_{u_{i,j-1}\wedge U_{A}}^{u_{i,j}\wedge U_{A}}
(1-2\Lambda_{v}(\xi_{v}))^{-1} dB_{v}
\\ \nonumber
&+o(\vert B_{u_{i,j}\wedge U_{A}}-B_{u_{i,j-1}\wedge U_{A}}\vert +(u_{i,j}-u_{i,j-1}))
+O(u_{i,j}-u_{i,j-1}).
\end{align*}
The sum 
\begin{displaymath}
\sum_{j=1}^{N_{i}}
(\rho_{u_{i,j}\wedge U_{A}}-\rho_{u_{i,j-1}\wedge U_{A}})^{2}
\end{displaymath}
behaves as $i\to +\infty$ like
\begin{displaymath}
\sum_{j=1}^{N_{i}}
\Big(
(1-2\Lambda_{u_{i,j-1}\wedge U_{A}}(\xi_{u_{i,j-1}\wedge U_{A}}))^{-1}
(B_{u_{i,j}\wedge U_{A}}-B_{u_{i,j-1}\wedge U_{A}})-
\int_{u_{i,j-1}\wedge U_{A}}^{u_{i,j}\wedge U_{A}}
(1-2\Lambda_{v}(\xi_{v}))^{-1} dB_{v}
\Big)^{2}.
\end{displaymath}
The expectation of the quantity above equals
\begin{multline*}
\sum_{j=1}^{N_{i}}
\E\Big[\int_{u_{i,j-1}\wedge U_{A}}^{u_{i,j}\wedge U_{A}} 
\big(
\dfrac{\partial}{\partial y}(\Psi^{B}_{v})^{-1}(\xi_{v})
-
\dfrac{\partial}{\partial y}(\Psi^{B}_{u_{i,j-1}\wedge U_{A}})^{-1}(\xi_{u_{i,j-1}\wedge U_{A}})
\big)^{2}
dv\Big]
\\=
o\Big(\sum_{j=1}^{N_{i}}(u_{i,j}-u_{i,j-1})\Big)
=o(1).
\qedhere
\end{multline*}
\end{proof}

\begin{rem}
The process $(\xi_{u})_{u\geq 0}$ has a decomposition into a sum of a local martingale and a process with $0$ quadratic variation, both adapted to the Brownian filtration $(\mathcal{F}^{B}_{u})_{u\geq 0}$. Following Föllmer's terminology \cite{Follmer81DirichletProc}, it is a Dirichlet process. However, it is believed not to be a semi-martingale \cite{HuWarren00BBFlow}, which would mean that
$(\rho_{u})_{u\geq 0}$  has an infinite total variation. The reason for that would be that the terms
$o(\vert B_{u_{i,j}}-B_{u_{i,j-1}}\vert
+(u_{i,j}-u_{i,j-1}))$ in \eqref{EqIto} are not
$O((B_{u_{i,j}}-B_{u_{i,j-1}})^{2})$, since the flow
$(\Psi_{u}^{B})_{u\geq 2}$ is not $\mathcal{C}^{2}$ in space.
One could push up to showing that $(\rho_{u})_{u\geq 0}$ is locally
$3/4-\varepsilon$ Hölder continuous. We believe that this 
$3/4-\varepsilon$ is optimal.
\end{rem}

Next are some elementary properties of the LRM $(X_{t})_{t\geq 0}$
(see Definition \ref{DefLSRM}).

\begin{prop}
\label{PropLSRM}
The following properties hold.
\begin{enumerate}
\item A.s., $X_{t}$ is defined for all $t\geq 0$.
\item A.s., for any $\alpha\in (0,1/2)$, the process
$(X_{t})_{t\geq 0}$ is locally $\alpha$-Hölder continuous.
\item Let $L_{t}$ be the occupation profile at time $t$, defined by
\eqref{EqLt}.
Then $(L_{t}(x)-L_{0}(x))_{x\in\R}$ is the occupation density of 
$X$ on time-interval $[0,t]$, that is to say, for any
$f:\R\rightarrow\R$ bounded,
\begin{displaymath}
\int_{0}^{t}f(X_{s}) ds =
\int_{\R}f(x)(L_{t}(x)-L_{0}(x)) dx.
\end{displaymath}
\item 
(Strong Markov property).
Let $T_{0}$ be a stopping time for the natural filtration
$(\mathcal{F}^{X}_{t})_{t\geq 0}$ of 
$(X_{t})_{t\geq 0}$. Then
$(X_{T_{0}+t})_{t\geq 0}$ is distributed as a Linearly Reinforced Motion starting from $X_{T_{0}}$, with
initial occupation profile $L_{T_{0}}$.
\item The process $(X_{t})_{t\geq 0}$ is recurrent,
that is to say, for all $t_{0}\geq 0$, the process will visit a.s. all points after $t_{0}$.
\item Let $x_{1}<x_{2}\in\R$. Then
\begin{equation}
\label{EqWhichHitFirst1}
\mathbb{P}(\text{After time}~t_{0}, X_{t}~\text{hits}~x_{2}~
\text{before}~x_{1}\vert \mathcal{F}^{X}_{t_{0}}, 
x_{1}<X_{t_{0}}<x_{2})\geq \dfrac{1}{2}
\end{equation}
is equivalent to
\begin{equation}
\label{EqWhichHitFirst2}
\int_{X_{t_{0}}}^{x_{2}} L_{t_{0}}(x)^{-2} dx\leq
\int_{x_{1}}^{X_{t_{0}}} L_{t_{0}}(x)^{-2} dx.
\end{equation}
More precisely, let
$y_{1}<0$ and $y_{2}>0$. Let
$U_{y_{1}}^{\downarrow}$ be the first time the drifted Brownian motion
$B_{u}-u$ hits $y_{1}$ and 
$U_{y_{2}}^{\uparrow}$ the first time 
$B_{u}+u$, hits $y_{2}$, with $B_{0}=0$.
Then
\begin{equation}
\label{EqUupdown}
\mathbb{P}(\text{After time}~t_{0}, X_{t}~\text{hits}~x_{2}~
\text{before}~x_{1}\vert \mathcal{F}^{X}_{t_{0}}, 
x_{1}<X_{t_{0}}<x_{2})=
\mathbb{P}(U_{y_{2}}^{\uparrow}<U_{y_{1}}^{\downarrow}),
\end{equation}
where
\begin{displaymath}
y_{1}=\int_{x_{1}}^{X_{t_{0}}} L_{t_{0}}(x)^{-2} dx,
\qquad
y_{2}=\int_{X_{t_{0}}}^{x_{2}} L_{t_{0}}(x)^{-2} dx.
\end{displaymath}
\item Let $t>0$ and consider $(t_{i,j})_{0\leq j\leq N_{i},i\geq 0}$
a deterministic family such that 
\begin{displaymath}
0=t_{i,0}<t_{i,1}<\dots<t_{i,N_{i}-1}<t_{N_{i}}=u
\end{displaymath}
and
\begin{displaymath}
\lim_{i\to +\infty}\max_{1\leq j\leq N_{i}}(t_{i,j}-t_{i,j-1})=0.
\end{displaymath}
Then
\begin{displaymath}
\lim_{i\to +\infty}\sum_{j=1}^{N_{i}}
(X_{t_{i,j}}-X_{t_{i,j-1}})^{2}
=\int_{0}^{t}L_{s}(X_{s}) ds
\end{displaymath}
in probability.
Let $(R_{t})_{t\geq 0}$ be the process
\begin{displaymath}
R_{t}=X_{t}-\int_{0}^{t} L_{s}(X_{s})^{2}d_{s}B_{u(s)},
\end{displaymath}
where $u(\cdot)$ is the inverse time-change of \eqref{EqTimeChange}.
Then, for a family $(t_{i,j})$ as above,
\begin{displaymath}
\lim_{i\to +\infty}\sum_{j=1}^{N_{i}}
(R_{t_{i,j}}-R_{t_{i,j-1}})^{2}
=0
\end{displaymath}
in probability.
\end{enumerate}
\end{prop}

\begin{proof}
(1): This is equivalent to
\begin{displaymath}
\int_{0}^{+\infty}
L_{0}(X_{u})^{3}(1-2\Lambda_{u}(\xi_{u}))^{-\frac{3}{2}} 
du = +\infty~\text{a.s.}
\end{displaymath}
Fix $y_{1}<y_{2}\in\R$. Since $L_{0}$ is positive bounded away from $0$ on $[S_{0}^{-1}(y_{2}),S_{0}^{-1}(y_{1})]$, it is enough to show that
\begin{displaymath}
\int_{0}^{+\infty}
\1_{y_{1}<\xi_{u}<y_{2}}(1-2\Lambda_{u}(\xi_{u}))^{-\frac{3}{2}} 
du = +\infty~\text{a.s.}
\end{displaymath}
Using the elementary properties of occupation densities, one can show that the above integrals equals
on $[S_{0}^{-1}(y_{2}),S_{0}^{-1}(y_{1})]$, it is enough to show that
\begin{displaymath}
\int_{y_{1}<y<y_{2}}\int_{0}^{+\infty}
(1-2\Lambda_{u}(\xi_{u}))^{-\frac{3}{2}} d_{u}\Lambda_{u}(y)
dy .
\end{displaymath}
Applying the identity \eqref{EqIDLocTime}, get that it equals in turn
\begin{displaymath}
\int_{y_{1}<y<y_{2}}\int_{0}^{+\infty}
\exp(3\L_{u}(y)) d_{u}\L_{u}(y)dy=
\dfrac{1}{3}\int_{y_{1}<y<y_{2}}(\exp(3\L_{+\infty}(y))-1) dy.
\end{displaymath}
Conclude using that a.s.,
$\forall y\in\R$, $\L_{+\infty}(y)=
\lim_{u\to +\infty}\L_{u}(y)=+\infty$
\cite{BassBurdzy99StochBiff,HuWarren00BBFlow}.

(2): This follows from the local Hölder continuity of $(\xi_{u})_{u\geq 0}$
and the fact that we perform $\mathcal{C}^{1}$ changes of scale and time.

(3): Use that
\begin{eqnarray*}
\int_{0}^{t}f(X_{s}) ds&=&
\int_{0}^{u(t)}f(S_{0}^{-1}(\xi_{v})) 
L_{0}(S_{0}^{-1}(\xi_{v}))^{3}
(1-2\Lambda_{v}(\xi_{v}))^{-\frac{3}{2}}dv
\\&=&
\int_{\R}\int_{0}^{u(t)}
f(S_{0}^{-1}(y)) 
L_{0}(S_{0}^{-1}(y))^{3}
(1-2\Lambda_{v}(y))^{-\frac{3}{2}}
d_{v}\Lambda_{v}(y) dy
\\&=&\int_{\R} f(S_{0}^{-1}(y)) 
L_{0}(S_{0}^{-1}(y))^{3}
((1-2\Lambda_{u(t)}(y))^{-\frac{1}{2}}-1) dy
\\&=&
\int_{\R} f(x) L_{0}(x)
((1-2\Lambda_{u(t)}(S_{0}(x)))^{-\frac{1}{2}}-1) dx.
\end{eqnarray*}

(4): Let $U_{0}=u(T_{0})$. It is a stopping time for the driving Brownian motion $(B_{u})_{u\geq 0}$.
Let $\tilde{\xi}_{u}=(\Psi^{B\circ\theta_{U_{0}}}_{u})^{-1}
((B\circ\theta_{U_{0}})_{u})$. The process
$(\tilde{\xi}_{u})_{u\geq 0}$ has the same law as $(\xi_{u})_{u\geq 0}$.
Moreover,
\begin{displaymath}
\xi_{U_{0}+u}=(\Psi^{B}_{U_{0}})^{-1}(\tilde{\xi}_{u}+B_{U_{0}}).
\end{displaymath}
Let
\begin{displaymath}
S_{T_{0}}(x)=\int_{X_{T_{0}}}^{x}L_{T_{0}}(r)^{-2} dr.
\end{displaymath}
We have that
\begin{displaymath}
X_{T_{0}+t}=S_{0}^{-1}(\xi_{u(T_{0}+t)})=
(\Psi^{B}_{U_{0}}\circ S_{0})^{-1}
(\tilde{\xi}_{u(T_{0}+t)-U_{0}}+B_{U_{0}}).
\end{displaymath}
Moreover,
\begin{displaymath}
(\Psi^{B}_{U_{0}}\circ S_{0})^{-1}
(0+B_{U_{0}})= X_{T_{0}}
\end{displaymath}
and, following \eqref{EqSpaceDeriv} and \eqref{EqIDLocTime},
\begin{displaymath}
\dfrac{d}{d x}
\Psi^{B}_{U_{0}}\circ S_{0}(x)=
L_{0}(x)^{-2}(1-2\Lambda_{U_{0}}(S_{0}(x)))=L_{T_{0}}(x)^{-2}.
\end{displaymath}
Thus, 
\begin{displaymath}
(\Psi^{B}_{U_{0}}\circ S_{0})^{-1}
(y+B_{U_{0}})=S_{T_{0}}^{-1}(y).
\end{displaymath}
Finally,
\begin{displaymath}
u(T_{0}+t)-U_{0}=\int_{T_{0}}^{T_{0}+t}L_{s}(X_{s})^{-3} ds.
\end{displaymath}
So we get (5).

(5): This follows from the recurrence of $(\xi_{u})_{u\geq 0}$.

(6): Since we have the Markov property, it is enough to show it for
$t_{0}=0$. Then, if $X_{0}\in(x_{1},x_{2})$,
\begin{eqnarray*}
\mathbb{P}(X_{t}~\text{hits}~x_{2}~
\text{before}~x_{1})&=&
\mathbb{P}(\xi_{u}~\text{hits}~S_{0}(x_{2})~
\text{before}~S_{0}(x_{1}))
\\&=&
\mathbb{P}(B_{u}~\text{meets}~\Psi^{B}_{u}\circ S_{0}(x_{2})~
\text{before}~\Psi^{B}_{u}\circ S_{0}(x_{1})),
\end{eqnarray*}
which is exactly \eqref{EqUupdown}.

(7): The proof is similar to that of (2) and (3) in Proposition \ref{PropAdditional}. One has to apply a time-change to go from
$(\xi_{u})_{u\geq 0}$ to
$(X_{t})_{t\geq 0}$, and thus, considers
$(B_{u})_{u\geq 0}$ at random stopping times rather than at fixed times.
Note that, in the time change \eqref{EqTimeChange},
\begin{displaymath}
L_{t}(X_{t}) dt = 
\Big(\dfrac{d}{dx}S_{0}(X_{t})\Big)^{-2}
(1-2\Lambda_{u}(\xi_{u}))^{-2} du.
\qedhere
\end{displaymath}
\end{proof}

\begin{rem}
\label{RemName}
The equivalence between \eqref{EqWhichHitFirst1} and \eqref{EqWhichHitFirst2} emphasizes the reinforcement property. Indeed, the motion tends to drift towards the places it has already visited a lot. Yet it is recurrent. Property (8) gives a decomposition of $(X_{t})_{t\geq 0}$ as a local martingale plus an adapted process with zero quadratic variation. As for $(\xi_{u})_{u\geq 0}$, we believe that the LRM $(X_{t})_{t\geq 0}$ is not a semi-martingale.
\end{rem}

\section{The VRJP-related random environment and its convergence}
\label{SecEnvir}

It was shown in \cite{SabotTarres2015VRJPSSHSM} that on general electrical networks the VRJP has the same law as a time-change of a non-reinforced Markov jump process in an environment with random conductances. 
This is stated and proved in full generality in 
Theorem 2 in \cite{SabotTarres2015VRJPSSHSM}.
One can also find the expression of the mixing measure
in Theorem 2 in \cite{SabotTarresZeng2017RandomSchrod}.
Here we will only give a statement in our one-dimensional setting, which is simpler.
In dimension one, the expression of the mixing measure has been already given in Theorem 1.1 in \cite{DavisVolkov2002VRJP}.

\begin{prop}[Davis-Volkov \cite{DavisVolkov2002VRJP},
Sabot-Tarrès \cite{SabotTarres2015VRJPSSHSM}]
\label{PropMelange}

Let $n\in\N$. 
Denote $\N^{\ast}=\N\setminus\{ 0\}$.
Let
$(V^{(n)-}(x))_{x\in 2^{-n}\N^{\ast}}$ and 
$(V^{(n)+}(x))_{x\in 2^{-n}\N^{\ast}}$ be two independent families of independent real random variables, where $V^{(n)\sigma}(x)$,
$\sigma\in\lbrace-1,+1\rbrace$, is distributed according to
\begin{displaymath}
2^{\frac{n}{2}-1}\pi^{-\frac{1}{2}}
(L_{0}(\sigma x)L_{0}(\sigma (x-2^{-n})))^{\frac{1}{2}}
\exp\left(-2^{n}L_{0}(\sigma x)L_{0}(\sigma (x-2^{-n}))
\sinh(v/2)^{2}+v/2\right)dv.
\end{displaymath}
Define $(\U^{(n)-}(x))_{x\in 2^{-n}\N}$ and 
$(\U^{(n)+}(x))_{x\in 2^{-n}\N}$ by
\begin{displaymath}
\U^{(n)-}(0)=\U^{(n)+}(0)=0,
\qquad
\U^{(n)\sigma}(x)=\sum_{i=1}^{2^{n} x}
V^{(n)\sigma}(2^{-n}i),
\sigma\in\lbrace -1,+1\rbrace, x\in\N^{\ast}.
\end{displaymath}
Set
\begin{equation}
\label{EqUn}
\U^{(n)}(x)=
\left\lbrace
\begin{array}{ll}
0 & \text{if}~x=0, \\ 
\U^{(n)+}(x) & \text{if}~x\in 2^{-n}\N^{\ast}, \\ 
\U^{(n)-}(\vert x\vert) & \text{if}~
-x\in 2^{-n}\N^{\ast}.
\end{array} 
\right.
\end{equation}
Let $(Z_{q}^{(n)})_{0\leq q\leq q^{(n)}_{\rm max}}$ be the continuous-time process on $2^{-n}\Z$, which, conditional on the random environment
$(\U^{(n)}(x))_{x\in 2^{-n}\Z}$, is a (non-reinforced) Markov jump process, started from $0$, with transition rate from $x\in 2^{-n}\Z$ to
$x+\sigma 2^{-n}$, $\sigma\in\lbrace -1,+1\rbrace$, equal to
\begin{equation}
\label{EqRateU}
2^{2n-1}\dfrac{L_{0}(x+\sigma 2^{-n})}{L_{0}(x)}
e^{-\U^{(n)}(x+\sigma 2^{-n})+\U^{(n)}(x)}.
\end{equation}
$q^{(n)}_{\rm max}\in (0,+\infty)$ is the time when the process explodes to infinity, whenever this happens. Otherwise
$q^{(n)}_{\rm max}=+\infty$.
Let $\lambda^{(n)}_{q}(x)$ be the local times of 
$Z_{q}^{(n)}$:
\begin{displaymath}
\lambda^{(n)}_{q}(x)=2^{n}\int_{0}^{q}\1_{Z_{r}^{(n)}=x} dr.
\end{displaymath}
Define the change of time
\begin{displaymath}
q^{(n)}(t)=\inf\Big\lbrace q\geq 0\Big\vert 2^{n}
\sum_{x\in 2^{-n}\Z}
\big((L_{0}(x)^{2}+2\lambda^{(n)}_{q}(x))^{\frac{1}{2}}-L_{0}(x)\big)
\geq t
\Big\rbrace.
\end{displaymath}
Then the family of time changed processes
\begin{displaymath}
\Big(Z^{(n)}_{q^{(n)}(t)\wedge q^{(n)}_{\rm max}},
(L_{0}(x)^{2}+
2\lambda^{(n)}_{q^{(n)}(t)\wedge q^{(n)}_{\rm max}}(x))^{\frac{1}{2}}\Big)
_{x\in 2^{-n}\Z, t\geq 0}
\end{displaymath}
has same distribution as the VRJP
\begin{displaymath}
(X^{(n)}_{t\wedge t^{(n)}_{\rm max}},
L^{(n)}_{t\wedge t^{(n)}_{\rm max}}(x))
_{x\in 2^{-n}\Z, t\geq 0}.
\end{displaymath}
\end{prop}

\begin{rem}
Theorem 2 in \cite{SabotTarres2015VRJPSSHSM} is given for finite graphs. To reduce to it, one has to take $A>0$ and consider only 
$x\in I_{A}^{(n)}:=2^{-n}\Z\cap [-A,A]$. The law corresponding to 
Theorem 2 in \cite{SabotTarres2015VRJPSSHSM} is for the random variables
\begin{displaymath}
\bar{\mathcal{U}}^{(n)}_{A}(x)=
-\mathcal{U}^{(n)}(x)
+\dfrac{1}{\operatorname{Card} I_{A}^{(n)}}
\sum_{x'\in I_{A}^{(n)}}
\mathcal{U}^{(n)}(x'),
\qquad x\in I_{A}^{(n)},
\end{displaymath}
so that
\begin{displaymath}
\sum_{x\in I_{A}^{(n)}}
\bar{\mathcal{U}}^{(n)}_{A}(x)=0.
\end{displaymath}
The density of the random vector
$(\bar{\mathcal{U}}^{(n)}_{A}(x))
_{x\in I_{A}^{(n)}}$
on the subspace
$$\Big\{ (\bar{u}(x))_{x\in I_{A}^{(n)}}
\in\R^{I_{A}^{(n)}}
\Big\vert \sum_{x\in I_{A}^{(n)}}
\bar{u}(x)=0\Big\}$$
is given by
\begin{multline*}
e^{\bar{u}(0)-\frac{1}{2}\bar{u}(\min I_{A}^{(n)})
-\frac{1}{2}\bar{u}(\max I_{A}^{(n)})}
~~
\smashoperator{
\prod_{\substack{
x\in 2^{-n}\Z\\
-A\leq x-2^{-n}\leq x\leq A
}}}
~~
2^{\frac{n}{2}-1}\pi^{-\frac{1}{2}}
(L_{0}(x)L_{0}(x-2^{-n}))^{\frac{1}{2}}
\\\times
\exp\big( 
-2^{n}L_{0}(x)L_{0}(x-2^{-n})\sinh((\bar{u}(x)-\bar{u}(x-2^{-n}))/2)^{2}
\big)
\\=
(2\pi)^{-N}
e^{\bar{u}(0)}
D^{(n)}_{A}(L_{0},\bar{u})^{\frac{1}{2}}
~~\smashoperator{
\prod_{\substack{
x\in 2^{-n}\Z\\
-A\leq x-2^{-n}\leq x\leq A
}}}~~
\exp\big( 
-2^{n}L_{0}(x)L_{0}(x-2^{-n})\sinh((\bar{u}(x)-\bar{u}(x-2^{-n}))/2)^{2}
\big),
\end{multline*}
where $2N+1=\operatorname{Card} I_{A}^{(n)}$ and
\begin{eqnarray*}
D^{(n)}_{A}(L_{0},\bar{u})&=&
~~
e^{-\bar{u}(\min I_{A}^{(n)})-\bar{u}(\max I_{A}^{(n)})}
\smashoperator{
\prod_{\substack{
x\in 2^{-n}\Z\\
-A\leq x-2^{-n}\leq x\leq A
}}}
~~
2^{n-1}
L_{0}(x)L_{0}(x-2^{-n})
\\&=&
\smashoperator{
\prod_{\substack{
x\in 2^{-n}\Z\\
-A\leq x-2^{-n}\leq x\leq A
}}}
~~
2^{n-1}
L_{0}(x)L_{0}(x-2^{-n})
e^{\bar{u}(x)+\bar{u}(x-2^{-n})}.
\end{eqnarray*}
As explained in
\cite{SabotTarres2015VRJPSSHSM} in the \textit{Nota Bene} (2) just below the statement of Theorem 2, the factor
$D^{(n)}_{A}(L_{0},\bar{u})$ in case of general electrical networks is replaced by a partition function on spanning trees with weighted edges. Here in the one-dimensional setting there is only one spanning tree including all the edges, and the weight of an edge $\{x-2^{n},x\}$ is
$$2^{n-1}L_{0}(x)L_{0}(x-2^{-n})
e^{\bar{u}(x)+\bar{u}(x-2^{-n})}.$$
\end{rem}
\begin{rem}
The random variable
$\exp(-V^{(n)\sigma}(x))$ follows an inverse Gaussian distribution with density
\begin{displaymath}
\1_{z>0}\dfrac{2^{\frac{n-1}{2}}
(L_{0}(\sigma x)L_{0}(\sigma (x-2^{-n})))^{\frac{1}{2}}
}
{(2\pi z^{3})^{\frac{1}{2}}}
\exp\Big(-2^{n-2}L_{0}(\sigma x)L_{0}(\sigma (x-2^{-n})
\big(z^{\frac{1}{2}}
-z^{-\frac{1}{2}}\big)^{2}\Big).
\end{displaymath}
The distribution that appears in 
Theorem 1.1 in \cite{DavisVolkov2002VRJP} is also an inverse Gaussian, and it is given for random variables corresponding to 
$\exp(-V^{(n)\sigma}(x))$. However, the parameters differ, as the VRJP there is parametrized differently.
\end{rem}

We will show that the random environment
$(\U^{(n)}(x))_{x\in 2^{-n}\Z}$ converges as $n\to +\infty$ to the process $(\U(x))_{x\in\R}$ introduced
in \eqref{EqU}. Out of this we deduce that
the process $(Z^{(n)}_{q})_{0\leq q\leq q^{(n)}_{\rm max}}$ has a limit in law $(Z_{q})_{q\geq 0}$, which condition on the environment
$(\U(x))_{x\in\R}$, is a Markov diffusion on $\R$.
Then, we conclude that the VRJP 
$(X^{(n)}_{t})_{0\leq t\leq t^{(n)}_{\rm max}}$ converges to a time-change of $(Z_{q})_{q\geq 0}$.

\begin{lemma}
\label{LemConvEnv}
The family of random processes
$(\U^{(n)}(x))_{x\in 2^{-n}\Z}$ defined in
\eqref{EqUn} converges in law,
as $n\to +\infty$, to the process $(\U(x))_{x\in\R}$,
defined by \eqref{EqU}. $\U^{(n)}$ is considered to be interpolated linearly outside $ 2^{-n}\Z$. The space of continuous functions $\R\rightarrow\R$ is considered to be endowed with the topology of uniform convergence on compact intervals.
\end{lemma}

\begin{proof}
For $K>0$, let $\mathsf{V}_{K}$ be a random variable with density
\begin{displaymath}
\dfrac{K^{\frac{1}{2}}}{(2\pi)^{\frac{1}{2}}}
\exp(-K\sinh(v/2)^{2} + v/2).
\end{displaymath}
Then $(K/2)^{\frac{1}{2}}(\mathsf{V}_{K}-K^{-1})$ converges in law as
$K\to +\infty$ to a standard centered Gaussian 
$\mathcal{N}(0,1)$. Indeed, the density of
$(K/2)^{\frac{1}{2}}(\mathsf{V}_{K}-K^{-1})$ is
\begin{displaymath}
\dfrac{1}{(2\pi)^{\frac{1}{2}}}
\exp\big(-K\sinh((2K)^{-\frac{1}{2}}z+(2K)^{-1})^{2}
+((2K)^{-\frac{1}{2}}z+(2K)^{-1})\big),
\end{displaymath}
which converges to the density of $\mathcal{N}(0,1)$. Moreover,
\begin{equation}
\label{EqUpBound}
\dfrac{1}{(2\pi)^{\frac{1}{2}}}
\exp\big(-K\sinh((2K)^{-\frac{1}{2}}z+(2K)^{-1})^{2}
+((2K)^{-\frac{1}{2}}z+(2K)^{-1})\big)
\leq
\dfrac{e^{(4K)^{-1}}}{(2\pi)^{\frac{1}{2}}}\exp(-z^{2}/2),
\end{equation}
where we just used that $\sinh(a)^{2}\geq a^{2}$.
Thus we have a stronger convergence. For any $f$ non-negative measurable function $\R \rightarrow\R$, such that $f$ is integrable for
$\mathcal{N}(0,1)$,
\begin{displaymath}
\E\big[f\big((K/2)^{\frac{1}{2}}(\mathsf{V}_{K}-K^{-1})\big)\big]
\leq \dfrac{e^{(4K)^{-1}}}{(2\pi)^{\frac{1}{2}}}
\int_{\R} f(z) e^{-z^{2}/2} dz,
\end{displaymath}
and by dominated convergence,
\begin{displaymath}
\lim_{K\to +\infty}
\E\big[f\big((K/2)^{\frac{1}{2}}(\mathsf{V}_{K}-K^{-1})\big)\big]
=\dfrac{1}{(2\pi)^{\frac{1}{2}}}
\int_{\R} f(z) e^{-z^{2}/2} dz.
\end{displaymath}
In particular, 
\begin{equation}
\label{EqMeanVar}
\E[\mathsf{V}_{K}] = K^{-1} + o(K^{-1}),
\qquad
\operatorname{Var}(\mathsf{V}_{K})=
2K^{-1} + o(K^{-1}).
\end{equation}
Let $\bar{\Phi}$ be the survival function of $\mathcal{N}(0,1)$:
\begin{equation}
\label{EqUpBound2}
\bar{\Phi}(A)=
\dfrac{1}{(2\pi)^{\frac{1}{2}}}\int_{A}^{+\infty} e^{-z^{2}/2} dz
\stackrel{A>>1}{=}O(e^{-A^{2}/2}).
\end{equation}
Then, by \eqref{EqUpBound},
\begin{equation}
\label{EqUpBound3}
\mathbb{P}(\vert \mathsf{V}_{K}-K^{-1}\vert\geq v)
\leq e^{(4K)^{-1}}\bar{\Phi}((K/2)^{\frac{1}{2}}v).
\end{equation}

Let be
\begin{displaymath}
\mathcal{E}^{(n)}(x)=
\E[\U^{(n)}(2^{-n}\lfloor 2^{n}x\rfloor)], \qquad
\mathcal{M}^{(n)}(x)=
\U^{(n)}(2^{-n}\lfloor 2^{n}x\rfloor)-\mathcal{E}^{(n)}(x),
\end{displaymath}
\begin{displaymath}
\mathcal{A}^{(n)}(x)=
\operatorname{Var}(\U^{(n)}(2^{-n}\lfloor 2^{n}x\rfloor)).
\end{displaymath}
$(\mathcal{M}^{(n)}(x))_{x\geq 0}$ and
$(\mathcal{M}^{(n)}(x)^{2}-\mathcal{A}^{(n)}(x))_{x\geq 0}$ are martingales.
The identities \eqref{EqMeanVar} implies that
$(\mathcal{E}^{(n)}(x))_{x\geq 0}$ and 
$(\mathcal{A}^{(n)}(x)/2)_{x\geq 0}$
converge to 
$(S_{0}(x))_{x\geq 0}$ uniformly on compact subsets. To conclude that
$(\mathcal{M}^{(n)}(x))_{x\geq 0}$ converges in law to
$(\sqrt{2}W\circ S_{0}(x))_{x\geq 0}$, and thus
$(\U^{(n)}(x))_{x\geq 0}$ converges in law to
$(\sqrt{2}W\circ S_{0}(x) + S_{0}(x))_{x\geq 0}$, we can apply a martingale functional Central Limit Theorem, the Theorem 1.4, Section 7.1 in
\cite{EthierKurtz1986Markov}. For this we need additionally to check that
for any $A>0$,
\begin{displaymath}
\lim_{n\to +\infty}
\E[\sup_{0\leq x\leq A}(\mathcal{M}^{(n)}(x)
-\mathcal{M}^{(n)}(x^{-}))^{2}]=0.
\end{displaymath}
Note that
\begin{eqnarray*}
\lim_{n\to +\infty}&\E &[\sup_{0\leq x\leq A}(\mathcal{M}^{(n)}(x)
-\mathcal{M}^{(n)}(x^{-}))^{2}]\\
&=&
\lim_{n\to +\infty}
\E[\sup_{x\in 2^{-n}\Z\cap [0,A]}(V^{(n)+}(x)
-\E[V^{(n)+}(x)])^{2}]\\
&=&
\lim_{n\to +\infty}
\E[\sup_{x\in 2^{-n}\Z\cap [0,A]}(V^{(n)+}(x)
-2^{-n}L_{0}(x)^{-1}L_{0}(x-2^{-n})^{-1})^{2}]\\
&=&
\lim_{n\to +\infty}
\int_{0}^{+\infty}
\mathbb{P}\big(
\sup_{x\in 2^{-n}\Z\cap [0,A]}\vert 
V^{(n)+}(x)
-2^{-n}L_{0}(x)^{-1}L_{0}(x-2^{-n})^{-1}\vert \geq v^{\frac{1}{2}}
\big) dv
\\
&=&
\lim_{n\to +\infty}
\int_{0}^{+\infty}
\big(1-\prod_{\substack{x\in 2^{-n}\Z\\\cap [0,A]}}
\big(1-
\mathbb{P}(
\vert V^{(n)+}(x)
-2^{-n}L_{0}(x)^{-1}L_{0}(x-2^{-n})^{-1}\vert \geq v^{\frac{1}{2}}
)
\big)\big)dv,
\end{eqnarray*}
and with \eqref{EqUpBound3} and \eqref{EqUpBound2} we get that
\begin{displaymath}
\lim_{n\to +\infty}
\int_{0}^{+\infty}
\big(1-\prod_{\substack{x\in 2^{-n}\Z\\\cap [0,A]}}
\big(1-
\mathbb{P}(
\vert V^{(n)+}(x)
-2^{-n}L_{0}(x)^{-1}L_{0}(x-2^{-n})^{-1}\vert \geq v^{\frac{1}{2}}
)
\big)\big)dv=0.
\end{displaymath}

The case $x\leq 0$ is similar.
\end{proof}

Recall that $(\lambda_{q}(x))_{x\in\R, q\geq 0}$ denotes the family of local times of $(Z_{q})_{q\geq 0}$.

\begin{prop}
\label{PropConvZq}
Consider the random environments $(\U^{(n)}(x))_{x\in 2^{-n}\Z}$
and the random processes $(Z^{(n)}_{q})_{0\leq q\leq q^{(n)}_{\rm max}}$,
with local times 
$(\lambda^{(n)}_{q}(x))_{x\in 2^{-n}\Z, 0\leq q\leq q^{(n)}_{\rm max}}$,
introduced in Proposition \ref{PropMelange}. As $n\to +\infty$,
$q^{(n)}_{\rm max}\to +\infty$ in probability, and the process
\begin{displaymath}
(\U^{(n)}(x),Z^{(n)}_{q\wedge q^{(n)}_{\rm max}}, 
\lambda^{(n)}_{q\wedge q^{(n)}_{\rm max}}(x))_{x\in 2^{-n}\Z, q\geq 0}
\end{displaymath}
converges in law to
\begin{displaymath}
(\U(x),Z_{q}(x),\lambda_{q}(x))_{x\in\R, q\geq 0}.
\end{displaymath}
We interpolate $2^{-n}\Z$-valued processes linearly, and use for
$\lambda^{(n)}_{q\wedge q^{(n)}_{\rm max}}(x)$ and
$\lambda_{q}(x)$ the topology of uniform convergence on compact subsets of $\R\times[0,+\infty)$.
\end{prop}

\begin{proof}
The idea is to "embed" the processes 
$(Z^{(n)}_{q})_{0\leq q\leq q^{(n)}_{\rm max}}$ for different values of $n$ inside a Brownian motion, scale-changed.
Let $(\beta_{s})_{s\geq 0}$ be a standard Brownian motion started from $0$, with a family of local times denoted
$(\ell^{\beta}_{s}(\varsigma))_{\varsigma\in\R, s\geq 0}$. Take
$(\U^{(n)}(x))_{x\in 2^{-n}\Z}$ independent from
$(\beta_{s})_{s\geq 0}$. Define the change of scale
$\mathcal{S}^{(n)}: 2^{-n}\Z \rightarrow \R$ by
$\mathcal{S}^{(n)}(0)=0$ and for $x\in 2^{-n}\Z$, $x\neq 0$,
$\mathcal{S}^{(n)}(x)$ equal to
\begin{displaymath}
\operatorname{sgn}(x)
2^{-n}\sum_{i=1}^{2^{n}\vert x\vert}
L_{0}(\operatorname{sgn}(x)2^{-n}i)^{-1}
L_{0}(\operatorname{sgn}(x)2^{-n}(i-1))^{-1}
e^{\U^{(n)}(\operatorname{sgn}(x)2^{-n}i)
+\U^{(n)}(\operatorname{sgn}(x)2^{-n}(i-1))}.
\end{displaymath}
Consider the time change
\begin{displaymath}
s^{(n)}(q)=\inf\Big\lbrace s\geq 0 \Big\vert
2^{-n}\sum_{x\in 2^{-n}\Z}
L_{0}(x)^{2}e^{-2\U^{(n)}(x)}
\ell^{\beta}_{s}(\mathcal{S}^{(n)}(x))
\geq q\Big\rbrace.
\end{displaymath}
Conditional on $(\U^{(n)}(x))_{x\in 2^{-n}\Z}$, the time changed Brownian motion
$(\beta_{s^{(n)}(q)})_{q\geq 0}$ is a Markov nearest neighbor jump process on
$\mathcal{S}^{(n)}(2^{-n}\Z)$, and the jump rate from 
$\mathcal{S}^{(n)}(x)$, $x\in 2^{-n}\Z$, to
$\mathcal{S}^{(n)}(x+\sigma 2^{-n})$,
$\sigma\in\{-1,1\}$, equals
\begin{displaymath}
2^{n}L_{0}(x)^{-2}e^{2\U^{(n)}(x)}
\dfrac{1}{2}
\vert \mathcal{S}^{(n)}(x+\sigma 2^{-n})-
\mathcal{S}^{(n)}(x)\vert^{-1}
=
2^{2n-1}\dfrac{L_{0}(x+\sigma 2^{-n})}{L_{0}(x)}
e^{-\U^{(n)}(x+\sigma 2^{-n})+\U^{(n)}(x)},
\end{displaymath}
which is exactly \eqref{EqRateU}.
Then one can construct
$Z^{(n)}_{q}$ and $\lambda^{(n)}_{q}(x)$ as
\begin{displaymath}
Z^{(n)}_{q} = (\mathcal{S}^{(n)})^{-1}(\beta_{s^{(n)}(q)}),
\qquad 
\lambda^{(n)}_{q}(x)=
L_{0}(x)^{2}e^{-2\U^{(n)}(x)}
\ell^{\beta}_{s^{(n)}(q)}(\mathcal{S}^{(n)}(x))
.
\end{displaymath}

Similarly, take $(\U(x))_{x\in\R}$ independent from
$(\beta_{s})_{s\geq 0}$. Consider the change of scale
\begin{displaymath}
\mathcal{S}(x)=\int_{0}^{x}L_{0}(r)^{-2}e^{2\U(r)} dr,
\end{displaymath}
and the change of time
\begin{equation}
\label{Eqsq}
s(q)=\inf\Big\lbrace s\geq 0\Big\vert
\int_{\R}
L_{0}(x)^{2}e^{-2\U(x)}
\ell^{\beta}_{s}(\mathcal{S}(x)) dx
\geq q
\Big\rbrace.
\end{equation}
One can construct
$Z_{q}$ and $\lambda_{q}(x)$ as
\begin{displaymath}
Z_{q} = \mathcal{S}^{-1}(\beta_{s(q)}),
\qquad 
\lambda_{q}(x)=
L_{0}(x)^{2}e^{-2\U(x)}
\ell^{\beta}_{s(q)}(\mathcal{S}(x))
.
\end{displaymath}
The convergence of $\U^{(n)}$ to $\U$
(Lemma \ref{LemConvEnv}) implies then the other convergences.
\end{proof}

Let be the time change
\begin{displaymath}
q(t)=\inf\Big\lbrace q\geq 0\Big\vert 
\int_{x\in \R}
\big((L_{0}(x)^{2}+2\lambda_{q}(x))^{\frac{1}{2}}-L_{0}(x)\big) dx
\geq t
\Big\rbrace.
\end{displaymath}
This is the same time-change as in \eqref{EqTC1}. Set
\begin{displaymath}
X^{\ast}_{t} = Z_{q(t)},
\qquad
L^{\ast}_{t}(x)= (L_{0}(x)^{2}+2\lambda_{q(t)}(x))^{\frac{1}{2}}.
\end{displaymath}

\begin{lemma}
\label{LemTimeChangeLSRM}

The function $t\mapsto q(t)$ is a.s. an increasing  diffeomorphism of
of $[0,+\infty)$. The space-time process
$(L^{\ast}_{t}(x)-L_{0}(x))_{x\in\R, t\geq 0}$ is the family of local times of $(X^{\ast}_{t})_{t\geq 0}$, that is to say for any 
$f$ bounded measurable function,
\begin{displaymath}
\int_{0}^{t} f(X^{\ast}_{s}) ds=
\int_{\R} f(x) (L^{\ast}_{t}(x)-L_{0}(x)) dx.
\end{displaymath}
\end{lemma}

\begin{proof}
For the first point, one needs to check that
\begin{displaymath}
\lim_{q\to +\infty}
\int_{x\in \R}
\big((L_{0}(x)^{2}+2\lambda_{q}(x))^{\frac{1}{2}}-L_{0}(x)\big) dx =+\infty.
\end{displaymath}
But actually, a.s. for all $x\in\R$,
$\lim_{q\to +\infty}\lambda_{q}(x)=+\infty$.

If we differentiate the time change $t\mapsto q(t)$, we get
\begin{displaymath}
dt = (L_{0}(Z_{q})^{2}+2\lambda_{q}(Z_{q}))^{-\frac{1}{2}}dq.
\end{displaymath}
Thus,
\begin{eqnarray*}
\int_{0}^{t} f(X^{\ast}_{s}) ds &=&
\int_{0}^{q(t)} f(Z_{r})
(L_{0}(Z_{r})^{2}+2\lambda_{r}(Z_{r}))^{-\frac{1}{2}} dr
\\&=&\int_{\R} \int_{0}^{q(t)}f(x)
(L_{0}(x)^{2}+2\lambda_{r}(x))^{-\frac{1}{2}}
d_{r}\lambda_{r}(x) dx
\\&=&\int_{\R} f(x)((L_{0}(x)^{2}+2\lambda_{q(t)}(x))^{\frac{1}{2}}
- L_{0}(x))dx
\\&=&\int_{\R} f(x) (L^{\ast}_{t}(x)-L_{0}(x)) dx,
\end{eqnarray*}
which is our second point.
\end{proof}

Combing Proposition \ref{PropMelange} and Proposition \ref{PropConvZq}, one immediately gets that the VRJP has a limit in law which is a time change of $(Z_{q})_{q\geq 0}$:

\begin{prop}
\label{ThmConv1}
As $n\to +\infty$, $t^{(n)}_{\rm max}\to +\infty$ in probability, and the VRJP
\begin{displaymath}
(X^{(n)}_{t\wedge t^{(n)}_{\rm max}},
L^{(n)}_{t}(x))_{x\in 2^{-n}\Z, t\geq 0}
\end{displaymath}
converges in law to
\begin{displaymath}
(X^{\ast}_{t},
L^{\ast}_{t}(x))_{x\in \R, t\geq 0},
\end{displaymath}
where we interpolate $L^{(n)}_{t}(x)$ linearly outside
$x\in 2^{-n}\Z$.
\end{prop}

Now let us recall how to obtain an ERRW as a mixture of random walks.
The statement below combines Theorem 1 in \cite{SabotTarres2015VRJPSSHSM},
which relates the ERRW and the VRJP, and Theorem 2 in \cite{SabotTarres2015VRJPSSHSM}, which relates the VRJP and a mixture of random walks.

\begin{prop}[Sabot-Tarrès \cite{SabotTarres2015VRJPSSHSM}]
\label{PropERRWVRJP}
Let $(\gamma(x-2^{-n},x))_{x\in 2^{-n}\Z}$ be independent random variables where $\gamma(x-2^{-n},x)$ has the distribution
$\Gamma(2^{n-1}L_{0}(x-2^{-n})L_{0}(x),1)$. Let
$(\widehat V^{(n)-}(x))_{x\in 2^{-n}\N^{\ast}}$ and 
$(\widehat V^{(n)+}(x))_{x\in 2^{-n}\N^{\ast}}$ be conditionally 
on $(\gamma(x-2^{-n},x))_{x\in 2^{-n}\Z}$
two independent families of independent real random variables, where 
$\widehat V^{(n)\sigma}(x)$,
$\sigma\in\lbrace-1,+1\rbrace$, has conditional distribution
\begin{displaymath}
(2\pi)^{-\frac{1}{2}}
(\gamma(x-\sigma 2^{-n},x))^{\frac{1}{2}}
\exp\left(-2\gamma(x-\sigma 2^{-n},x)
\sinh(v/2)^{2}+v/2\right)dv.
\end{displaymath}
Define $(\widehat\U^{(n)-}(x))_{x\in 2^{-n}\N}$ and 
$(\widehat\U^{(n)+}(x))_{x\in 2^{-n}\N}$ by
\begin{displaymath}
\widehat\U^{(n)-}(0)=\widehat\U^{(n)+}(0)=0,
\qquad
\widehat\U^{(n)\sigma}(x)=\sum_{i=1}^{2^{n} x}
\widehat V^{(n)\sigma}(2^{-n}i),
\sigma\in\lbrace -1,+1\rbrace, x\in\N^{\ast}.
\end{displaymath}
Set
\begin{equation*}
\widehat\U^{(n)}(x)=
\left\lbrace
\begin{array}{ll}
0 & \text{if}~x=0, \\ 
\widehat\U^{(n)+}(x) & \text{if}~x\in
2^{-n}\N^{\ast}, \\ 
\widehat\U^{(n)-}(\vert x\vert) & \text{if}~-x\in
2^{-n}\N^{\ast}.
\end{array} 
\right.
\end{equation*}
Consider the discrete time random walk on
$2^{-n}\Z$, started from $0$, in the random environment
$(\gamma(x-2^{-n},x),\widehat\U^{(n)}(x))_{x\in 2^{-n}\Z}$,
with conditional transition probabilities from $x$ to $x+\sigma 2^{-n}$,
$\sigma\in\{-1,1\}$, proportional to
\begin{displaymath}
\gamma(x,x+\sigma 2^{-n})
e^{-(\widehat\U^{(n)}(x)+\widehat\U^{(n)}(x+\sigma 2^{-n}))}.
\end{displaymath}
Then, averaged by the environment, it
has same distribution as the ERRW 
$(\widehat{Z}^{(n)}_{k})_{k\geq 0}$ of
Proposition \ref{PropConvERRW}.
\end{prop}

The following elementary convergence in probability holds.

\begin{lemma}
\label{LemTildeL0}
(1) Let $A>0$. Let $S_{0}$ be the change of scale \eqref{EqS0}, with $x_{0}=0$.
Then
\begin{displaymath}
\sup_{x\in [-A,A]\cap 2^{-n}\Z}
\Big\vert 
\operatorname{sgn}(x)
\dfrac{1}{2}\sum_{i=1}^{2^{n}\vert x\vert}
\gamma(2^{-n}i-\operatorname{sgn}(x)2^{-n},2^{-n}i)^{-1}-S_{0}(x)\Big\vert
\end{displaymath}
converges in probability to $0$ as $n\to +\infty$.

(2) Let $A>0$. We have that
\begin{displaymath}
\sup_{x\in [-A,A]\cap 2^{-n}\Z}
\Big\vert 
\operatorname{sgn}(x)
2^{-2n+1}\sum_{i=1}^{2^{n}\vert x\vert}
\gamma(2^{-n}i-\operatorname{sgn}(x)2^{-n},2^{-n}i)
-\int_{0}^{x}L_{0}(r)^{2} dr\Big\vert
\end{displaymath}
converges in probability to $0$ as $n\to +\infty$.
\end{lemma}

\begin{proof}
(1): By the elementary properties of gamma distributions, 
\begin{eqnarray*}
\E[\gamma(x-2^{-n},x)^{-1}]=
\dfrac{\Gamma(\cdot - 1)}
{\Gamma(\cdot)}(2^{n-1}L_{0}(x-2^{-n})L_{0}(x))=
(2^{n-1}L_{0}(x-2^{-n})L_{0}(x)-1)^{-1},
\\
\operatorname{Var}(\gamma(x-2^{-n},x)^{-1})=
\left(\dfrac{\Gamma(\cdot - 2)}{\Gamma(\cdot)}-
\dfrac{\Gamma(\cdot - 1)^{2}}{\Gamma(\cdot)^{2}}\right)
(2^{n-1}L_{0}(x-2^{-n})L_{0}(x))
=O(2^{-3n}),
\end{eqnarray*}
where $\Gamma$ is the Euler's Gamma function, and
$\Gamma(\cdot - k)(a)$ stands for $\Gamma(a - k)$, so as to shorten the expressions above.
From Doob's maximal inequality follows that
\begin{multline*}
\E\bigg[\sup_{x\in [0,A]\cap 2^{-n}\Z}
\Big\vert \dfrac{1}{2}\sum_{i=1}^{2^{n}x}
\gamma(2^{-n}(i-1),2^{-n}i)^{-1}-
\E[\gamma(2^{-n}(i-1),2^{-n}i)^{-1}]\Big\vert^{2}\bigg]
\\\leq
4 \sum_{i=1}^{\lfloor 2^{n} A\rfloor}
\operatorname{Var}(\gamma(2^{-n}(i-1),2^{-n}i)^{-1})=
O(2^{-2n}).
\end{multline*}
Moreover,
\begin{displaymath}
\lim_{n\to +\infty}\sup_{x\in [-A,A]\cap 2^{-n}\Z}
\Big\vert \dfrac{1}{2}\sum_{i=1}^{2^{n}\vert x\vert}
\E[\gamma(2^{-n}i-\operatorname{sgn}(x)2^{-n},2^{-n}i)^{-1}]-S_{0}(x)\Big\vert=0.
\end{displaymath}

(2): The proof is similar to that of (1), using that
\begin{eqnarray*}
\E[2^{-2n+1}\gamma(x-2^{-n},x)]&=&2^{-n}L_{0}(x-2^{-n})L_{0}(x),
\\
\operatorname{Var}(2^{-2n+1}\gamma(x-2^{-n},x))
&=&2^{-3n+1}L_{0}(x-2^{-n})L_{0}(x),
\end{eqnarray*}
and applying Doob's maximal inequality.
\end{proof}

\begin{proofERRW}
Lemma \ref{LemTildeL0} (1) implies that
$(\widehat{\U}^{(n)}(x))_{x\in 2^{-n}\Z}$ converges in law, for the topology of uniform convergence on compacts, to 
$\U$ given by \eqref{EqU}. This can be proved similarly to
Lemma \ref{LemConvEnv}. Define the change of scale
$\widehat{\mathcal{S}}^{(n)}: 2^{-n}\Z \rightarrow\R$ by
$\widehat{\mathcal{S}}^{(n)}(0)=0$, and for $x\in 2^{-n}\Z$,
$x\neq 0$,
\begin{displaymath}
\widehat{\mathcal{S}}^{(n)}(x)=
\operatorname{sgn}(x)
\dfrac{1}{2}\sum_{i=1}^{2^{n}\vert x\vert}
\gamma(2^{-n}i-\operatorname{sgn}(x)2^{-n},2^{-n}i)^{-1}
e^{\widehat{\U}^{(n)}(2^{-n}i-\operatorname{sgn}(x)2^{-n})+\widehat{\U}^{(n)}(2^{n}i)}.
\end{displaymath}
Under this change of scale,
$(\widehat{\mathcal{S}}^{(n)}(\widehat{Z}^{(n)}_{k}))_{k\geq 0}$
conditional on the random environment
\\$(\gamma(x-2^{-n},x),\widehat\U^{(n)}(x))_{x\in 2^{-n}\Z}$ is a 
martingale.
Lemma \ref{LemTildeL0} (1) combined with the convergence of
$\widehat{\U}^{(n)}$ to $\U$, implies in turn that
$\widehat{\mathcal{S}}^{(n)}$ 
converges in law to $\mathcal{S}$ given by \eqref{EqNatScal}.
Let $(\widetilde{Z}^{(n)}_{q})_{q\geq 0}$ be an auxiliary process that has same trajectory as 
$(\widehat{Z}^{(n)}_{\lfloor 4^{n}q\rfloor})_{q\geq 0}$, but instead of
jumping at deterministic times in
$4^{-n}\N\setminus\{ 0\}$, jumps at independent exponential times with mean $4^{-n}$. Then the convergence of 
$(\widehat{Z}^{(n)}_{\lfloor 4^{n}q\rfloor})_{q\geq 0}$ 
is equivalent to that of 
$(\widetilde{Z}^{(n)}_{q})_{q\geq 0}$. As in the proof of Proposition
\ref{PropConvZq}, one can embed $(\widetilde{Z}^{(n)}_{q})_{q\geq 0}$
into a standard Brownian motion $(\beta_{s})_{s\geq 0}$.
$(\ell^{\beta}_{s}(\varsigma))_{\varsigma\in\R, s\geq 0}$ will denote the family of local times of the Brownian motion. 
We will also consider $(\beta_{s})_{s\geq 0}$ independent of the environment
$(\gamma(x-2^{-n},x),\widehat\U^{(n)}(x))_{x\in 2^{-n}\Z}$.
Let be the time change
\begin{multline*}
\mathsf{s}^{(n)}(q)=\\
\inf\Big\{s\geq 0\Big\vert
\dfrac{4^{-n}}{2}\sum_{x\in 2^{-n}\Z}
\big(
(\widehat{\mathcal{S}}^{(n)}(x+2^{-n})
-\widehat{\mathcal{S}}^{(n)}(x))^{-1}
+
(
\widehat{\mathcal{S}}^{(n)}(x)
-\widehat{\mathcal{S}}^{(n)}(x-2^{-n})
)^{-1}
\big)
\ell^{\beta}_{s}(\widehat{\mathcal{S}}^{(n)}(x))
\geq q
\Big\}.
\end{multline*}
Then one can take
\begin{displaymath}
\widetilde{Z}^{(n)}_{q}=
(\widehat{\mathcal{S}}^{(n)})^{-1}
(\beta_{\mathsf{s}^{(n)}(q)}).
\end{displaymath}
to conclude, we need the convergence in probability of
$(\mathsf{s}^{(n)}(q))_{q\geq 0}$, uniformly on compact subsets, to
$(s(q))_{q\geq 0}$ given by \eqref{Eqsq}. Note that
\begin{displaymath}
(\widehat{\mathcal{S}}^{(n)}(x)
-\widehat{\mathcal{S}}^{(n)}(x-2^{-n}))^{-1} = 
2\gamma(x-2^{-n},x)
e^{-\widehat{\U}^{(n)}(x-2^{-n})-\widehat{\U}^{(n)}(x)}.
\end{displaymath}
The convergence of $\widehat{\U}^{(n)}$ to $\U$, of
$\widehat{\mathcal{S}}^{(n)}$ to
$\mathcal{S}$ and Lemma
\ref{LemTildeL0} (2) imply the desired convergence.
\end{proofERRW}

\section{Convergence of the VRJP to the Linearly Reinforced Motion}
\label{SecConv}

In this section we prove that 
the Vertex Reinforced Jump Processes converges in law to a Linearly Reinforced Motion constructed using the Bass-Burdzy flow (Section \ref{SecBassBurdzy}). To this end, we will make appear something that looks like a Bass-Burdzy flow in discrete. We also use that we already have a limit obtained as a time-changed Markov diffusion in a random environment (Proposition \ref{ThmConv1}).

Define the scale functions $x\mapsto S^{(n)}_{t}(x)$ by
\begin{displaymath}
S^{(n)}_{0}(x)=
\left\lbrace
\begin{array}{l}
0~~~~~\text{if}~x=0, \\ 
2^{-n}\sum_{i=1}^{2^{n}x}
L_{0}(2^{-n}i)^{-1}L_{0}(2^{-n}(i-1))^{-1}
~~~~~\text{if}~x\in 2^{-n}\Z\cap(0,+\infty),\\  
-2^{-n}\sum_{i=1}^{2^{n}\vert x\vert}
L_{0}(-2^{-n}i)^{-1}L_{0}(-2^{-n}(i-1))^{-1}
~~~~~\text{if}~x\in 2^{-n}\Z\cap(-\infty,0),\\ 
(2^{-n}\lceil 2^{n}x\rceil-x)
S^{(n)}_{0}(2^{-n}\lfloor 2^{n}x\rfloor) 
+
(x-2^{-n}\lfloor 2^{n}x\rfloor)
S^{(n)}_{0}(2^{-n}\lceil 2^{n}x\rceil)
~~~~~\text{if}~x\not\in 2^{-n}\Z,
\end{array}
\right.
\end{displaymath}
and
\begin{displaymath}
\dfrac{\partial}{\partial t} S^{(n)}_{t}(x)=
\left\lbrace
\begin{array}{ll}
0 & \text{if}~x=X^{(n)}_{t}, \\ 
-L^{(n)}_{t}(X^{(n)}_{t})^{-2}
L^{(n)}_{t}(X^{(n)}_{t}+2^{-n})^{-1} & 
\text{if}~x\geq X^{(n)}_{t}+2^{-n},
 \\ 
+L^{(n)}_{t}(X^{(n)}_{t})^{-2}
L^{(n)}_{t}(X^{(n)}_{t}-2^{-n})^{-1} & 
\text{if}~x\leq X^{(n)}_{t}-2^{-n}, \\ 
-(x-2^{-n}\lfloor 2^{n}x\rfloor)
L^{(n)}_{t}(X^{(n)}_{t})^{-2}
L^{(n)}_{t}(X^{(n)}_{t}+2^{-n})^{-1} & 
\text{if}~x\in(X^{(n)}_{t},X^{(n)}_{t}+2^{-n}), \\ 
+(2^{-n}\lceil 2^{n}x\rceil-x)
L^{(n)}_{t}(X^{(n)}_{t})^{-2}
L^{(n)}_{t}(X^{(n)}_{t}-2^{-n})^{-1} & 
\text{if}~x\in(X^{(n)}_{t}-2^{-n},X^{(n)}_{t}).
\end{array}
\right. 
\end{displaymath}

\begin{rem}
\label{Rem}
$x\mapsto S_{t}^{(n)}(x)$ is a strictly increasing function. 
$S_{t}^{(n)}$ has been constructed in a way so as to always have, for $x\in 2^{-n}\Z$,
\begin{displaymath}
S_{t}^{(n)}(x)-S_{t}^{(n)}(x-2^{-n})=
2^{-n}L_{t}^{(n)}(x)^{-1}L_{t}^{(n)}(x-2^{-n})^{-1}.
\end{displaymath}
In particular,
\begin{eqnarray*}
S_{t}^{(n)}(+\infty)-S_{t}^{(n)}(-\infty)&=&
2^{-n}\sum_{i\in\Z}L_{t}^{(n)}(2^{-n}i)^{-1}
L_{t}^{(n)}(2^{-n}(i-1))^{-1}\\&\leq &
2^{-n}\sum_{i\in\Z}L_{0}(2^{-n}i)^{-1}
L_{0}(2^{-n}(i-1))^{-1}.
\end{eqnarray*}
Moreover,
\begin{equation}
\label{EqLimScale}
\lim_{n\to +\infty}S_{0}^{(n)}(x)=
\int_{0}^{x}L_{0}(r)^{-2} dr=
S_{0}(x).
\end{equation}
Condition \eqref{EqCond} ensures that $S_{0}(+\infty)=+\infty$ and
$S_{0}(-\infty)=-\infty$. However, for finite $n$, we do not necessarily have $S_{0}^{(n)}(+\infty)=+\infty$ and
$S_{0}^{(n)}(-\infty)=-\infty$.
\end{rem}

Consider the change of time
\begin{displaymath}
du^{(n)}(t)=\dfrac{L_{t}(X^{(n)}_{t}-2^{-n})+
L_{t}(X^{(n)}_{t}+2^{-n})}
{2L_{t}(X^{(n)}_{t})^{2}
L_{t}(X^{(n)}_{t}-2^{-n})
L_{t}(X^{(n)}_{t}+2^{-n})} dt,
\end{displaymath}
and the inverse time change
$t^{(n)}(u)$, for 
$u\in (0,u^{(n)}_{\rm max})
=(0,u^{(n)}(t^{(n)}_{\rm max}))\subseteq (0,+\infty)$.

\begin{lemma}
\label{LemMart}
The process
\begin{equation}
\label{EqDefMart}
(M^{(n)}_{u})_{u\geq 0}:=
(S_{t^{(n)}(u)\wedge t^{(n)}_{\rm max}}(X^{(n)}
_{t^{(n)}(u)\wedge t^{(n)}_{\rm max}}))_{u\geq 0}
\end{equation}
is a martingale with respect to its natural
filtration $(\mathcal{F}^{M^{(n)}}_{u})_{u\geq 0}$. 
It advances by jumps at discrete times. A.s.,
$u_{\rm max}^{(n)}=+\infty$.
Moreover, for $u_{1}>u_{0}\geq 0$,
\begin{displaymath}
\E[(M_{u_{1}}^{(n)}-M_{u_{0}}^{(n)})^{2}\vert 
\mathcal{F}^{M^{(n)}}_{u_{0}}]=
u_{1}-u_{0}.
\end{displaymath}
\end{lemma}

\begin{proof}
Given $u\in [0,u^{(n)}_{\rm max})$, $M^{(n)}$ will make a jump on
the infinitesimal time interval $(u,u+du)$ with infinitesimal probability
\begin{multline}
\label{EqInfJumpRatedu}
2^{2n-1}(L_{t^{(n)}(u)}(X^{(n)}_{t^{(n)}(u)}-2^{-n})+
L_{t^{(n)}(u)}(X^{(n)}_{t^{(n)}(u)}+2^{-n}))
\dfrac{d t^{(n)}(u)}{du}du
\\=4^{n}
L_{t^{(n)}(u)}(X^{(n)}_{t^{(n)}(u)})^{2}
L_{t^{(n)}(u)}(X^{(n)}_{t^{(n)}(u)}-2^{-n})
L_{t^{(n)}(u)}(X^{(n)}_{t^{(n)}(u)}+2^{-n})
du.
\end{multline}
Conditional that the jump occurs, it will be of height
\begin{displaymath}
+2^{-n}
L_{t^{(n)}(u)}(X^{(n)}_{t^{(n)}(u)})^{-1}
L_{t^{(n)}(u)}(X^{(n)}_{t^{(n)}(u)}+2^{-n})^{-1}
\end{displaymath}
with probability
\begin{displaymath}
\dfrac{L_{t^{(n)}(u)}(X^{(n)}_{t^{(n)}(u)}+2^{-n})}
{L_{t^{(n)}(u)}(X^{(n)}_{t^{(n)}(u)}-2^{-n})+
L_{t^{(n)}(u)}(X^{(n)}_{t^{(n)}(u)}+2^{-n})},
\end{displaymath}
and of height
\begin{displaymath}
-2^{-n}
L_{t^{(n)}(u)}(X^{(n)}_{t^{(n)}(u)})^{-1}
L_{t^{(n)}(u)}(X^{(n)}_{t^{(n)}(u)}-2^{-n})^{-1}
\end{displaymath}
with probability
\begin{displaymath}
\dfrac{L_{t^{(n)}(u)}(X^{(n)}_{t^{(n)}(u)}-2^{-n})}
{L_{t^{(n)}(u)}(X^{(n)}_{t^{(n)}(u)}-2^{-n})+
L_{t^{(n)}(u)}(X^{(n)}_{t^{(n)}(u)}+2^{-n})}.
\end{displaymath}
So the expected height of the jump is $0$, and the expected height squared is
\begin{displaymath}
4^{-n}
L_{t^{(n)}(u)}(X^{(n)}_{t^{(n)}(u)})^{-2}
L_{t^{(n)}(u)}(X^{(n)}_{t^{(n)}(u)}-2^{-n})^{-1}
L_{t^{(n)}(u)}(X^{(n)}_{t^{(n)}(u)}+2^{-n})^{-1},
\end{displaymath}
which is exactly the inverse of the jump rate
\eqref{EqInfJumpRatedu}.

Let $(U_{N})_{N\geq 0}$ be the family of stopping times after performing $N$ jumps. We get that 
$(M^{(n)}_{u_{1}\wedge U_{N}}-M^{(n)}_{u_{0}\wedge U_{N}})_{N\geq 0}$
is an $\mathbb{L}^{2}$ convergent martingale and at the limit,
\begin{eqnarray*}
\E[(M_{u_{1}}^{(n)}-M_{u_{0}}^{(n)})^{2}\vert 
\mathcal{F}^{M^{(n)}}_{u_{0}}]&=&
\lim_{N\to +\infty}
\E[(M^{(n)}_{u_{1}\wedge U_{N}}-M^{(n)}_{u_{0}\wedge U_{N}})^{2}\vert 
\mathcal{F}^{M^{(n)}}_{u_{0}}]
\\&=&\lim_{N\to +\infty}
\E[u_{1}\wedge U_{N}-u_{0}\wedge U_{N}\vert\mathcal{F}^{M^{(n)}}_{u_{0}}]
\\&=&
\E[u_{1}\wedge u^{(n)}_{\rm max}-
u_{0}\wedge u^{(n)}_{\rm max}\vert\mathcal{F}^{M^{(n)}}_{u_{0}}]
\end{eqnarray*}
Since on the event $u^{(n)}_{\rm max}\in (u_{0},u_{1})$ we would have
$(M_{u_{1}}^{(n)}-M_{u_{0}}^{(n)})^{2}=+\infty$, this in particular means that it has probability $0$, and further that
$u_{\rm max}^{(n)}=+\infty$ a.s.
\end{proof}

We consider the process $((X^{\ast}_{t})_{t\geq 0},
(L^{\ast}_{t}(x))_{x\in \R, t\geq 0})$ 
obtained as a limit in law of the VRJP
$((X^{(n)}_{t})_{0\leq t\leq t^{(n)}_{\rm max}},
(L^{(n)}_{t}(x))_{x\in 2^{-n}\Z, 0\leq t\leq t^{(n)}_{\rm max}})$
in Theorem \ref{ThmConv1}. We define
\begin{displaymath}
\widetilde{S}_{t}^{\ast}(x)=
\int_{X^{\ast}_{t}}^{x}
L^{\ast}_{t}(r)^{-2} dr.
\end{displaymath}
$\widetilde{S}_{t}^{\ast -1}$ is the inverse diffeomorphism of 
$\widetilde{S}_{t}^{\ast}$ on $\R$. We define the time change
\begin{displaymath}
du^{\ast}(t)= 
L^{\ast}_{t}(X^{\ast}_{t})^{-3} dt,
\end{displaymath}
and $t^{\ast}(u)$ the inverse time change. 

\begin{lemma}
\label{LemUast}
A.s., $u^{\ast}(+\infty)= +\infty$.
\end{lemma}

\begin{proof}
\begin{displaymath}
u^{\ast}(+\infty)=\int_{0}^{+\infty}L^{\ast}_{t}(X^{\ast}_{t})^{-3} dt=+\infty.
\end{displaymath}
But the above integral equals
\begin{displaymath}
\int_{\R}\int_{0}^{+\infty}
L^{\ast}_{t}(x)^{-3} d_{t}L^{\ast}_{t}(x)dx=
\dfrac{1}{2}\int_{\R}L_{0}(x)^{-2} dx,
\end{displaymath}
which is $+\infty$ by \eqref{EqCond}.

\end{proof}

In discrete, we define
\begin{displaymath}
\widetilde{S}_{t}^{(n)}(x)=S_{t}^{(n)}(x)-
S_{t}^{(n)}(X^{(n)}_{t}),
\end{displaymath}
and $\widetilde{S}_{t}^{(n) -1}$ the inverse function on
$(S_{t}^{(n)}(-\infty)-
S_{t}^{(n)}(X^{(n)}_{t}),
S_{t}^{(n)}(+\infty)-
S_{t}^{(n)}(X^{(n)}_{t}))$.

From Theorem \ref{ThmConv1} immediately follows the following convergence result:

\begin{lemma}
\label{LemConvALot}
We have a joint convergence in law of processes
\begin{displaymath}
(X^{(n)}_{t}, L^{(n)}_{t}(x),
u^{(n)}(t), t^{(n)}(u),
\widetilde{S}_{t}^{(n)}(x),\widetilde{S}_{t}^{(n) -1}(y))
\end{displaymath}
towards
\begin{displaymath}
(X^{\ast}_{t}, L^{\ast}_{t}(x),
u^{\ast}(t), t^{\ast}(u),
\widetilde{S}_{t}^{\ast}(x),\widetilde{S}_{t}^{\ast -1}(y)).
\end{displaymath}
For $\widetilde{S}_{t}^{(n)}(x)$ and
$\widetilde{S}_{t}^{(n) -1}(y)$ we use the topology of uniform convergence on compact subsets of
$\R\times [0,+\infty)$. In particular,
$t^{(n)}_{\rm max}$ 
converges in probability towards $+\infty$, and, for any $t_{0}\geq 0$,
\begin{displaymath}
(\sup_{0\leq t\leq t_{0}}\widetilde{S}_{t}^{(n)}(-\infty),
\inf_{0\leq t\leq t_{0}}\widetilde{S}_{t}^{(n)}(+\infty))
\end{displaymath}
converges in probability towards $(-\infty,+\infty)$.
\end{lemma}

\begin{prop}
\label{PropConvBM}
The martingale $(M^{(n)}_{u})_{u\geq 0}$, introduced in
\eqref{EqDefMart}, converges in law to a standard Brownian motion started at $0$, $(B_{u})_{u\geq 0}$, in the Skorokhod topology.
\end{prop}

\begin{proof}
For $A>0$, let $T^{(n)}_{A}$ be the first time $X^{(n)}_{t}$ exits from the interval $[-A,A]$. Define $(M_{u}^{(n,A)})_{u\geq 0}$ to be the process that coincides with 
$(M_{u}^{(n)})_{u\geq 0}$ on the time-interval
$[0,u^{(n)}(T^{(n)}_{A})]$, and after time
$u^{(n)}(T^{(n)}_{A})$ behaves like conditional independent standard Brownian motion started from $M_{u^{(n)}(T^{(n)}_{A})}^{(n)}$.
$(M_{u}^{(n,A)})_{u\geq 0}$ is constructed in a way such that it is a martingale started from $0$ and moreover,
$((M_{u}^{(n,A)})^{2}-u)_{u\geq 0}$ is a martingale too.
Furthermore, one has a uniform control on the size of the jump of
$(M_{u}^{(n,A)})_{u\geq 0}$. All of them are smaller than or equal to
\begin{displaymath}
2^{-n}\Big(\inf_{[-A-2^{-n},A+2^{-n}]}L_{0}\Big)^{-2},
\end{displaymath}
and, in particular,
\begin{displaymath}
\lim_{n\to +\infty}\E\Big[\sup_{u\geq 0}
(M_{u}^{(n,A)}-M_{u^{-}}^{(n,A)})^{2}\Big]=0.
\end{displaymath}
According to Theorem 1.4, Section 7.1 in \cite{EthierKurtz1986Markov}
(a martingale functional Central Limit Theorem),
$(M_{u}^{(n,A)})_{u\geq 0}$ converges in law as $n\to +\infty$ to a standard Brownian motion started from $0$. Now, $T^{(n)}_{A}$ converges in law to $T^{\ast}_{A}$, the first time $X^{\ast}_{t}$ exits
$[-A,A]$, and $u^{(n)}(T^{(n)}_{A})$ converges to
\begin{displaymath}
\int_{0}^{T^{\ast}_{A}}L^{\ast}_{t}(X^{\ast}_{t})^{-3} dt.
\end{displaymath}
In particular,
\begin{displaymath}
\lim_{u\to +\infty}\sup_{n\in\mathbb{N}}
\mathbb{P}(u^{(n)}(T^{(n)}_{A})\leq u)=0.
\end{displaymath}
Thus, $(M_{u}^{(n)})_{u\geq 0}$ converges in law to a Brownian motion, too.
\end{proof}

\begin{prop}
\label{PropMainBis}
The limit process
$((X^{\ast}_{t})_{t\geq 0},
(L^{\ast}_{t}(x))_{x\in \R, t\geq 0})$ obtained in
Proposition \ref{ThmConv1} has the same law as a Linearly Reinforced Motion $((\widehat{X}_{t})_{t\geq 0},(L_{t}(x))_{x\in \R, t\geq 0})$ started from $0$, with initial occupation profile $L_{0}$.
Consequently, one gets Theorem \ref{ThmMain} and \ref{ThmMixture}.
\end{prop}

\begin{proof}
From Lemma \ref{LemConvALot} and Proposition
\ref{PropConvBM}, the process
\begin{displaymath}
(X^{(n)}_{t}, L^{(n)}_{t}(x),
u^{(n)}(t), t^{(n)}(u),
\widetilde{S}_{t}^{(n)}(x),\widetilde{S}_{t}^{(n) -1}(y),M^{(n)}_{u})
\end{displaymath}
is tight, and therefore has a subsequential limit in law
\begin{equation}
\label{EqLimitLaw}
(X^{\ast}_{t}, L^{\ast}_{t}(x),
u^{\ast}(t), t^{\ast}(u),
\widetilde{S}_{t}^{\ast}(x),\widetilde{S}_{t}^{\ast -1}(y),B_{u}),
\end{equation}
where $(B_{u})_{u\geq 0}$ is a standard Brownian motion
started from $0$. Define
\begin{displaymath}
S^{\ast}_{t}(x)=\widetilde{S}_{t}^{\ast}(x)+B_{u^{\ast}(t)},
\end{displaymath}
and
\begin{displaymath}
\Psi^{\ast}_{u}(y)=S^{\ast}_{t^{\ast}(u)}\circ S^{\ast -1}_{0}(y)=
\widetilde{S}^{\ast}_{t^{\ast}(u)}\circ S^{ -1}_{0}(y) + B_{u},
\end{displaymath}
where $S_{0}$ is given by \eqref{EqLimScale}. $\Psi^{\ast}_{u}(y)$ is the limit (along the subsequence we consider) of
\begin{displaymath}
\Psi^{(n)}_{u}(y)=S^{(n)}_{t^{(n)}(u)}\circ S^{(n) -1}_{0}(y).
\end{displaymath}

We want to show that $(\Psi^{\ast}_{u})_{u\geq 0}$ is the Bass-Burdzy flow associated to $(B_{u})_{u\geq 0}$. We have, for $u<u^{(n)}_{\rm max}$ and
$y\in (S^{(n)}_{0}(-\infty),S^{(n)}_{0}(+\infty))$, that
\begin{multline*}
\dfrac{\partial}{\partial u}\Psi^{(n)}_{u}(y)=\\
\left\lbrace
\begin{array}{ll}
0 & \text{if}~y=M^{(n)}_{u}, \\ 
-\dfrac{2L^{(n)}_{t^{(n)}(u)}(X^{(n)}_{t^{(n)}(u)}-2^{-n})}
{L^{(n)}_{t^{(n)}(u)}(X^{(n)}_{t^{(n)}(u)}-2^{-n})
+L^{(n)}_{t^{(n)}(u)}(X^{(n)}_{t^{(n)}(u)}+2^{-n})} & 
\text{if}~(\Psi^{(n)}_{u})^{-1}(y)\geq
X^{(n)}_{t^{(n)}(u)}+2^{-n},\\ 
+\dfrac{2L^{(n)}_{t^{(n)}(u)}(X^{(n)}_{t^{(n)}(u)}+2^{-n})}
{L^{(n)}_{t^{(n)}(u)}(X^{(n)}_{t^{(n)}(u)}-2^{-n})
+L^{(n)}_{t^{(n)}(u)}(X^{(n)}_{t^{(n)}(u)}+2^{-n})} & 
\text{if}~(\Psi^{(n)}_{u})^{-1}(y)\leq
X^{(n)}_{t^{(n)}(u)}-2^{-n},
\end{array} 
\right.
\end{multline*}
and in all other cases, 
\begin{multline}
\label{EqBound}
-\dfrac{2L^{(n)}_{t^{(n)}(u)}(X^{(n)}_{t^{(n)}(u)}-2^{-n})}
{L^{(n)}_{t^{(n)}(u)}(X^{(n)}_{t^{(n)}(u)}-2^{-n})
+L^{(n)}_{t^{(n)}(u)}(X^{(n)}_{t^{(n)}(u)}+2^{-n})}
\leq\dfrac{\partial}{\partial u}\Psi^{(n)}_{u}(y)\\\leq
\dfrac{2L^{(n)}_{t^{(n)}(u)}(X^{(n)}_{t^{(n)}(u)}+2^{-n})}
{L^{(n)}_{t^{(n)}(u)}(X^{(n)}_{t^{(n)}(u)}-2^{-n})
+L^{(n)}_{t^{(n)}(u)}(X^{(n)}_{t^{(n)}(u)}+2^{-n})}.
\end{multline}
Since $L^{(n)}_{t^{(n)}(u)}$ converges, for $y$ away from $B_{u}$, 
\begin{displaymath}
\dfrac{\partial}{\partial u}\Psi^{\ast}_{u}(y)=
-\1_{y>B_{u}}+\1_{y<B_{u}}.
\end{displaymath}
\eqref{EqBound} and the convergence of local times implies that
$u\mapsto\Psi^{\ast}_{u}(y)$ is Lipschitz-continuous.
Thus, according to Theorem 2.3 in \cite{BassBurdzy99StochBiff},
$(\Psi^{\ast}_{u})_{u\geq 0}$ is the Bass-Burdzy flow associated to
$(B_{u})_{u\geq 0}$. 

Let
\begin{displaymath}
\xi^{\ast}_{u}=(\Psi^{\ast}_{u})^{-1}(B_{u}).
\end{displaymath}
We have that
\begin{displaymath}
X^{\ast}_{t}=
\widetilde{S}_{t}^{\ast -1}(0)=
S_{0}^{-1}(\xi^{\ast}_{u^{\ast}(t)}).
\end{displaymath}
Thus, $X^{\ast}_{t}$ follows the definition of a Linearly Reinforced Motion with driving Brownian motion
$(B_{u})_{u\geq 0}$ (Definition \ref{DefLSRM}). It follows that the limit law for \eqref{EqLimitLaw} is unique and we have the desired identity in law.
\end{proof}

\begin{prop}
\label{PropNormOccupMeas}

Let be a Linearly Reinforced Motion
$((X_{t})_{t\geq 0},
(L_{t}(x))_{x\in \R, t\geq 0})$ started from $0$, with initial occupation profile $L_{0}$. It is coupled with the random environment
$(\U(x))_{x\geq 0}$ (see \eqref{EqU}).
For any $x_{1},x_{2}\in \R$, 
\begin{displaymath}
\lim_{t\to +\infty}\dfrac{L_{t}(x_{2})}{L_{t}(x_{1})}=
\dfrac{L_{0}(x_{2})e^{-\U(x_{2})}}
{L_{0}(x_{1})e^{-\U(x_{1})}}~~\text{a.s.}
\end{displaymath}
Moreover, the convergence is a.s. uniform on compact subsets of
$\R^{2}$.
In particular, the random environment
$(\U(x))_{x\geq 0}$ is measurable with respect to
$(X_{t})_{t\geq 0}$.
\end{prop}

\begin{proof}
The measure
$L_{0}(x)^{2}e^{-2\U(x)} dx$ is finite and invariant for
$(Z_{q})_{q\geq 0}$. According the ergodic theorem for one-dimensional diffusions (Section 6.8 in \cite{ItoMcKean1974Diffusions}),
\begin{displaymath}
\lim_{q\to +\infty}\dfrac{1}{q}\lambda_{q}(x)= 
c_{1}L_{0}(x)^{2}e^{-2\U(x)}~~\text{a.s.},
\end{displaymath}
where
\begin{equation}
\label{EqNormInt}
c_{1}^{-1}=\int_{\R}L_{0}(r)^{2}e^{-2\U(r)} dr
=\int_{\R}e^{-2\sqrt{2}W(y)-2\vert y\vert} dy.
\end{equation}
For the uniform convergence, see \cite{VanZanten2003Erg}.
Then,
\begin{displaymath}
L_{t}(x)=
(L_{0}(x)^{2}+2\lambda_{q(t)}(x))^{\frac{1}{2}}
\sim\sqrt{2c_{1}}q(t)^{\frac{1}{2}}
L_{0}(x)e^{-\U(x)}.
\qedhere
\end{displaymath}
\end{proof}

\begin{rem}
The measure $L_{0}(x)e^{-\U(x)}$ is not necessarily finite. We have that
\begin{displaymath}
\int_{\R}L_{0}(x)e^{-\U(x)} dx
=\int_{\R}\Big(\dfrac{d}{dy}S_{0}^{-1}(y)
\Big)^{-\frac{1}{2}}e^{-\sqrt{2}W(y)-\vert y\vert} dy,
\end{displaymath}
where $S_{0}^{-1}$ can be any increasing diffeomorphism from 
$\R$ to $\R$. The integral above being finite is a $0$-$1$ property, but there are examples where it is infinite. For that it is sufficient that
\begin{displaymath}
\int_{\R}\Big(\dfrac{d}{dy}S_{0}^{-1}(y)
\Big)^{-\frac{1}{2}}e^{-(1+\varepsilon)\vert y\vert} dy = +\infty.
\end{displaymath}
In the case when it is finite, the normalized occupation measure $\frac{1}{t}(L_{t}(x)-L_{0}(x)) dx$ converges a.s., in the weak topology of measures, to
\begin{displaymath}
c_{2} L_{0}(x)e^{-\U(x)} dx,
\end{displaymath}
where $c_{2}$ is a normalization factor.
\end{rem}

Next we give the large time behaviour of $(X_{t})_{t\geq 0}$. Actually, the leading order is given by the deterministic drift part in the random potential $2\U - 2\log(L_{0})$.

\begin{prop}
\label{PropAsymp}
Consider a Linearly Reinforced Motion
$((X_{t})_{t\geq 0},
(L_{t}(x))_{x\in \R, t\geq 0})$ started from $0$, with the initial occupation profile $L_{0}$ being equal to $1$ everywhere, except possibly a compact interval. Then,
\begin{displaymath}
\limsup_{t\to +\infty}\dfrac{X_{t}}{\log(t)}
=\dfrac{1}{3}~~\text{a.s.},
\qquad
\liminf_{t\to +\infty}\dfrac{X_{t}}{\log(t)}=-
\dfrac{1}{3}~~\text{a.s.}
\end{displaymath}
The mixture of diffusions $(Z_{q})_{q\geq 0}$ such that
$X_{t}=Z_{q(t)}$, with
\begin{displaymath}
dt = (L_{0}(Z_{q})^{2}+2\lambda_{q}(Z_{q}))^{-\frac{1}{2}}dq,
\end{displaymath}
satisfies
\begin{displaymath}
\limsup_{q\to +\infty}\dfrac{Z_{q}}{\log(q)}
=\dfrac{1}{6}~~\text{a.s.},
\qquad
\liminf_{q\to +\infty}\dfrac{Z_{q}}{\log(q)}=-
\dfrac{1}{6}~~\text{a.s.}
\end{displaymath}
\end{prop}

\begin{proof}
The measure $L_{0}(x)^{2}e^{-2\U(x)} dx$ is a finite invariant measure for $(Z_{q})_{q\geq 0}$. 
According to \cite{VanZanten2003Erg},
\begin{displaymath}
\lim_{q\to +\infty}\sup_{x\in\R}\vert q^{-1}\lambda_{q}(x)-
c_{1}L_{0}(x)^{2}e^{-2\U(x)}\vert = 0,
\end{displaymath}
where $c_{1}$ is given by \eqref{EqNormInt}. Thus,
\begin{displaymath}
t=\int_{\R}
((L_{0}(x)^{2}+2\lambda_{q(t)}(x))^{\frac{1}{2}} -L_{0}(x))dx
\sim \sqrt{2 c_{1}}
\left(\int_{\R}L_{0}(x)e^{-\U(x)} dx\right)q(t)^{\frac{1}{2}},
\end{displaymath}
and
\begin{equation}
\label{Eqlogqt}
\log(t)\sim \dfrac{1}{2}\log(q(t)).
\end{equation}
So we are left to determine
\begin{displaymath}
\limsup_{q\to +\infty}\dfrac{Z_{q}}{\log(q)}~~\text{and}~~
\liminf_{q\to +\infty}\dfrac{Z_{q}}{\log(q)}.
\end{displaymath}

Consider the natural scale function $\mathcal{S}$ of
$(Z_{q})_{q\geq 0}$, given by \eqref{EqNatScal}. We have that
\begin{equation}
\label{EqS1}
\mathcal{S}^{-1}(\varsigma)
\stackrel{+\infty}{\sim}\dfrac{1}{2}\log (\varsigma),
\qquad
\mathcal{S}^{-1}(\varsigma)
\stackrel{-\infty}{\sim}-\dfrac{1}{2}\log (\vert\varsigma\vert).
\end{equation}
$(\mathcal{S}(Z_{q}))_{q\geq 0}$ is a Brownian motion 
$(\beta_{s})_{s\geq 0}$ time-changed, with the time-change given by
\begin{displaymath}
ds=L_{0}(Z_{q})^{2}e^{-2\U(Z_{q})} dq,
\end{displaymath}
and the inverse time change
\begin{displaymath}
dq=L_{0}(\mathcal{S}^{-1}(\beta_{s}))^{-2}
e^{2\U(\mathcal{S}^{-1}(\beta_{s}))} ds.
\end{displaymath}
We have that for any $\alpha>1$, a.s. there is $K_{\alpha}>1$, such that
\begin{displaymath}
\forall \varsigma\in\R,~
K_{\alpha}^{-1}\vert \varsigma\vert^{\alpha^{-1}}\leq
L_{0}(\mathcal{S}^{-1}(\varsigma))^{-2}
e^{2\U(\mathcal{S}^{-1}(\varsigma))}\leq
K_{\alpha}\vert \varsigma\vert^{\alpha}.
\end{displaymath}
Then, using the Brownian scaling, we get that for $\alpha>1$ and some random $\widetilde{K}_{\alpha}>1$,
\begin{equation}
\label{EqErgqs}
\widetilde{K}_{\alpha}^{-1}s^{1+\frac{\alpha^{-1}}{2}}
\leq q(s)\leq 
\widetilde{K}_{\alpha}s^{1+\frac{\alpha}{2}},
\end{equation}
According the law of iterated logarithm,
\begin{displaymath}
\limsup_{s\to +\infty}\dfrac{\mathcal{S}(Z_{q(s)})}
{(2s\log\log(s))^{\frac{1}{2}}}=1,
\qquad
\liminf_{s\to +\infty}\dfrac{\mathcal{S}(Z_{q(s)})}
{(2s\log\log(s))^{\frac{1}{2}}}=-1.
\end{displaymath}
Combining with\eqref{EqS1} and \eqref{EqErgqs}, we get that
\begin{displaymath}
\limsup_{q\to +\infty}\dfrac{\log(Z_{q})}{\log(q)}=\dfrac{1}{6},
\qquad
\liminf_{q\to +\infty}\dfrac{\log(Z_{q})}{\log(q)}=-\dfrac{1}{6}.
\end{displaymath}
Combining with \eqref{Eqlogqt}, we get the result.
\end{proof}

\section*{Acknowledgements}

This work was supported by the French National Research Agency (ANR) grant
within the project MALIN (ANR-16-CE93-0003).

This work was partly supported by the LABEX MILYON (ANR-10-LABX-0070) of Université de Lyon, within the program "Investissements d'Avenir" (ANR-11-IDEX-0007) operated by the French National Research Agency (ANR).

TL acknowledges the support of Dr. Max Rössler, the Walter Haefner
Foundation and the ETH Zurich Foundation.

PT acknowledges the support of the National Science Foundation of China (NSFC), grant No. 11771293.

\bibliographystyle{alpha}
\bibliography{titusbibnew}

\end{document}